\documentclass[leqno,11pt]{amsart}

\usepackage{amsmath,amsfonts,amsthm,mathrsfs,amssymb}
\usepackage[all]{xy}
\usepackage{enumerate}
\usepackage{graphicx,epsfig}
\usepackage{mathtools}
\usepackage{bm}
\usepackage{tikz}
\usepackage{pgf}
\usepackage{epsfig}

\usetikzlibrary{matrix}

\theoremstyle{plain}
\newtheorem{thm}{Theorem}[section]

\newtheorem{lem}[thm]{Lemma}

\newtheorem{nt}[thm]{Notation}

\newtheorem{prop}[thm]{Proposition}
\newtheorem{ques}[thm]{Question}

\theoremstyle{definition}
\newtheorem{defn}[thm]{Definition}
\newtheorem{exmp}[thm]{Example}
\newtheorem{rem}[thm]{Remark}

\newcommand\dto{\dashrightarrow}

\newcommand\PP{\mathbb P}
\def\iP{{\mathbb P}}
\newcommand\ZZ{\mathbb Z}
\newcommand\NN{\mathbb N}
\newcommand\CC{\mathbb C}

\newcommand\ip{\mathfrak{p}}
\newcommand\IP{\mathfrak{P}}
\newcommand\iq{\mathfrak{q}}

\newcommand\ia{\mathfrak{a}}
\newcommand\im{\mathfrak{m}}

\newcommand\Cc{\mathcal C}
\newcommand{\EE}{\mathcal{E}}

\newcommand\Ss{\mathcal S}

\def\Z{{\mathcal Z}}
\def\R{{\mathcal R}}

\def\Fi{{\mathcal F}}

\def\cO{{\mathcal O}}

\def\C.{C_\bullet}
\def\F.{F_\bullet}
\def\k.{\mathcal{K}_{\bullet}}

\newcommand\inn{\mathfrak{n}}
\newcommand\ra{\rightarrow}
\newcommand\ann{\textnormal{ann}}

\newcommand\coker{\textnormal{coker}}
\newcommand\cohd{\textnormal{cd}}
\newcommand\fin{\textnormal{end}}

\newcommand\codim{\textnormal{codim}}

\newcommand\Spec{\textnormal{Spec}}

\newcommand\sym{\textnormal{Sym}}
\newcommand\indeg{\textnormal{indeg}}
\newcommand\Proj{\textnormal{Proj}}
\newcommand\Ext{\textnormal{Ext}}
\newcommand\Fitt{\textnormal{Fitt}}

\newcommand\tor{\textnormal{Tor}}
\newcommand\Tor{\textnormal{{Tor}}}

\newcommand\ext{\textnormal{Ext}}
\newcommand\Hom{\textnormal{Hom}}

\newcommand\reg{\textnormal{reg}}

\newcommand{\lr}[1]{\langle #1 \rangle}

\newcommand\Af{\mathfrak{A}}
\newcommand\Sc{\mathcal{S}}

\newcommand\MM{\mathbf{M}}
\newcommand\corank{\mathrm{corank}}

\def\Supp{\mathrm{Supp}}

\def\lra{\longrightarrow}

\begin{document}

\title[Fibers of rational maps and elimination matrices]{Fibers of rational maps and elimination matrices: an application oriented approach}

\author{Laurent Busé}
\address{Université Côte d'Azur, Inria, 2004 route des Lucioles, 06902 Sophia Antipolis, France.}
\email{laurent.buse@inria.fr}

\author{Marc Chardin}
\address{Institut de Mathématiques de Jussieu, CNRS \& Sorbonne Université, France}
\email{marc.chardin@imj-prg.fr or Marc.CHARDIN@cnrs.fr}

\dedicatory{Dedicated to David Eisenbud on the occasion of his seventy-fifth birthday.}

\begin{abstract} Parameterized algebraic curves and surfaces are widely used in geometric modeling and their manipulation is an important task in the processing of geometric models. In particular, the determination of the intersection loci between points, pieces of parameterized algebraic curves and pieces of algebraic surfaces is a key problem in this context. In this paper, we survey recent methods based on syzygies and blowup algebras for computing the image and the finite fibers of a curve or surface parameterization, more generally of a rational map. Conceptually, the main idea is to use elimination matrices, mainly built from syzygies, as representations of rational maps and to extract geometric informations from them. The construction and main properties of these matrices are first reviewed and then illustrated through several settings, each of them {highlighting} a particular feature of this approach that combines tools from commutative algebra, algebraic geometric and elimination theory.  
\end{abstract}

\maketitle

\section{Introduction}

In geometric modeling or closely related domains, parameterized curves or surfaces are used intensively. Actually, 2D and 3D geometric objects are often represented by assembling pieces of algebraic rational curves and surfaces that are called rational B\'ezier curves and surfaces. Typical examples goes from the letter fonts stored in a computer to a complex CAD model of mechanical pieces (see e.g.~\cite[Chapter 3]{CoxCBMS} and \cite{Farinbook,Patbook}). In this context, intersection problems between rational curves and surfaces are central questions to be solved. {An important problem} is to decide whether a point belongs to a given rational curve or surface. There is a huge literature on this topic with quite different types of techniques (see \cite[Chapter 5]{Patbook} and references therein). Among them, the development of algebraic methods to turn the parametric representation of an algebraic curve or surface into an implicit representation received a lot of attention. This so-called implicitization problem is quite useful, because the membership problem we just mentioned can be decided by means of an evaluation operation, which is much simpler. Mathematical tools for solving the implicitization problem goes back to the elimination theory as developed by Sylvester, Cayley, Dixon and others (see e.g.~\cite{Dixon,Goldman}). A more modern version, based on resultant theory, can be found in many recent papers and books (see e.g.~\cite{CLO98,GKZ94,Jou97}).

From a practical point of view, the implicitization of a rational curve or surface by means of polynomial equation(s), typically the implicit equation $F(x_0,x_1,x_2)=0$ of a rational planar curve, is not enough for solving intersection problems on geometric models. Indeed, as already mentioned, in this field geometric models are built from pieces of rational curves and surfaces. Therefore, given a point on a parameterized curve or surface it is necessary to determine its pre-image(s) via the parameterization in order to decide if it belongs to the piece that is used (see e.g.~\cite{Shen:2016:LNS:3045889.3064444}). It is clear that an implicit equation does not allow to do that directly as it only {detects} if a point belongs to the entire algebraic curve or surface. In what follows, we will focus on tools and methods from commutative algebra and algebraic geometry that have been developed in order to not only describe implicit equations of the image of a rational map, but also to analyze and determine the fibers of these rational maps, especially the finite fibers.

The setting we will focus on in this survey is the following. Suppose given a rational map $\psi:X \dasharrow \iP^{r-1}_k$, where $X$ will be typically a (product of) projective space(s) of dimension $\leq 3$ ($r=3,4$) over a field $k$, and assume that it is generically finite onto its image. Then, we would like to detect and compute the finite fibers of $\psi$. The tools that we will present are deeply rooted in elimination theory with a particular focus on elimination matrices. The simplest examples of such matrices are the famous Sylvester matrix, Dixon matrices \cite{Dixon} or Macaulay matrices \cite{Mac02,Jou97}. The key observation here is to lift the rational map $\psi$ as a projection from its graph $\Gamma$ to $\iP^{r-1}$, turning this way our problem into the study of a projection (elimination) map. More precisely, there is a diagram
\begin{equation*}
 \xymatrix@1{\Gamma\ \ar[d]_{\pi_1}\ar[rd]^{\pi_2}\ar@{^(->}[r] & X\times_{k}{\iP_k^{r-1}} \\  X  \ar@{-->}[r]^(.4){\psi} & \iP_k^{r-1}}	
\end{equation*}
and  we will describe how the fibers of $\pi_2$ can be analyzed by means of elimination matrices. 

\medskip

The paper is organized as follows. In Section \ref{sec:graph} we set notation and we introduce blowup algebras associated to a rational map $\psi$, namely the Rees and symmetric algebras, and recall their connection to the graph of $\psi$. Section \ref{sec:elimmat} is devoted to elimination techniques. We first show how Fitting ideals associated to some graded components of blowup algebras are connected to the image and fibers of a rational map. Then, we explain how the   choice of the graded components to be considered is governed by the control of the vanishing of some local cohomology modules. {This} section ends with an important result showing that for finite fibers the control of this vanishing can be done globally (and not for each point). 
 In Section \ref{sec:P1} we derive the first consequences of the above-mentioned methods in the case of curve parameterizations, i.e.~the case where the source is a projective line. Besides the computation of the fibers of such maps, we also provide an estimate of the Castelnuovo-Mumford regularity of rational curves as a by-product of our approach. 
Section \ref{sec:Pn-1} deals with the case of hypersurface parameterizations without base point, i.e.~morphisms from $\iP^{n-1}$ to $\iP^{n}$. Here, the emphasis is on the use of non linear equations of the Rees algebra in order to get more compact elimination matrices. Then, the case of surface parameterizations is discussed in Section \ref{sec:dim2}, where the presence of finitely many base points is considered. There is also a detailed discussion on the enumeration and determination of positive dimensional fibers of such parameterizations. Finally, the paper ends with Section \ref{sec:orthoproj} where the challenging case of 3-dimensional generically finite and dominant rational maps is treated. This setting is actually motivated by the computation of the euclidian distance of a point to a parameterized 3D surface, an important problem in geometric modeling. It will also be the occasion to illustrate how to deal with blowup algebras over multigraded rings.

\medskip

Works of David Eisenbud were very inspiring while exploring this subject and the title of his beautiful book named `The Geometry of Syzygies" fits perfectly our approach, even though he probably had in mind other strong and fruitful relations between geometry and syzygies while choosing this title.

\section{Graph of a rational map}\label{sec:graph}

Given a rational map, the determination of its image or of the parameters corresponding to one point of its image (the fiber) relies, at least in the algebraic approach that we present, on the choice of compactifications for the source and the target. It turns out in practice that a good choice for compactifying the source could help in speeding up the computations, this choice being adapted to the type of parametrization that is used. For this reason, in this survey we will focus on rational maps of the form
\begin{eqnarray}\label{eq:psi}
\psi: X & \dasharrow & \iP^{r-1}_k \\ \nonumber
x & \mapsto & 	(f_1 (x):\cdots :f_r(x)),
\end{eqnarray}
where $X$ is a toric variety given in terms of its Cox ring, over a field $k$. A very important case is when $X=\iP^{n-1}_k$ is itself a projective space, but choosing for $X$ a product of projective spaces is also often used. 

Recall that a Cox ring is an extension of the usual construction used in the case of a projective space or a product of such spaces \cite{CoxTV}. It is given by a grading on a polynomial ring $R$ by an abelian group $G$ (that corresponds to the Picard group of $X$) and a specific monomial ideal $B$ that defines the empty set in $X$. The points of $X$ (in the scheme sense, in other words irreducible reduced subvarieties) correspond to $G$-graded prime ideals and there is a one-to-one correspondence between subschemes of $X$ and $G$-graded $R$-ideals that are $B$-saturated. Coherent $G$-graded $R$-modules correspond to coherent sheaves on $X$ and graded pieces of their local cohomology with respect to $B$ correspond to sheaf cohomology of corresponding twists of this coherent sheaf, in a very similar way as in the case of a projective space.

Thus, the rational map $\psi$, as defined by \eqref{eq:psi}, corresponds to a tuple of homogeneous polynomials $(f_1,\ldots ,f_r)$ of {the} same degree $d\in G$. This tuple is uniquely determined, up to multiplication by an element in $k\setminus \{ 0\}$, assuming that the $f_i$'s have no common factor. To understand image and fibers of $\psi$, it is no surprise that the graph of $\psi$ plays a very important role. As we will now see, this graph corresponds to the notion of Rees algebra.

\medskip

\subsection{Graph and Rees algebra}
For an ideal $I$ in a ring $R$, the Rees algebra is defined as the subalgebra $\oplus_{t\geq 0} I^t T^t\subseteq R[T]$. From this description, it is clear that it is a domain if $R$ is a domain. When $I=(f_1,\ldots ,f_r)$ is a graded ideal in a graded ring $R$ and the $f_i$'s are of the same degree $d$, the Rees algebra admits a bigrading as follows. Let $S:=R[T_1,\ldots ,T_r]$ with the bigrading (i.e.~$G\times \ZZ$-grading) $\deg (T_i):=(0,1)$ and $\deg (x):=(\deg (x),0)$ for $x\in R$, then there is a bigraded onto map
\begin{eqnarray*}
S=R[T_1,\ldots ,T_r]& \rightarrow & \R_I :=\oplus_t I^t (td)\\
T_i & \mapsto  &f_i\in I(d)_0=I_d.	
\end{eqnarray*}
This grading gives $(\R_I)_{\mu ,t}=(I^t)_{\mu +td}$ for $t\geq 0$. Moreover, $\R_I =S/\IP$ for some $G\times \ZZ$-graded prime ideal $\IP$ whose elements are called the equations of the Rees algebra $\R_I$.

As a key property, the Rees algebra $\R_I$ defines the graph of $\psi$. To prove it, notice that the defining ideal of the Rees algebra $\IP\subseteq S$ contains the elements $G_{ij}:=f_iT_j-f_jT_i$ for any $i<j$. In particular, off $V(I)$, the Rees algebra coincides with the graph of $\psi$. Thus, the closure $\Gamma$ of the graph in $X \times \iP^{r-1}_k$ is the closure of the variety defined by the $G_{ij}$'s in $(X\setminus V(I))\times \iP^{r-1}_k$, which is the one defined by $\R_I$, as $\R_I$ is a domain. We notice that another more geometrical way to state this, is that the Koszul relations $G_{ij}$ {define} a subscheme of $X \times \iP^{r-1}_k=\Proj_{G\times \ZZ}(S)$ that contains the graph of $\psi$ (more precisely the Zariski closure $\Gamma$ of this graph) as an irreducible component. The following lemma is useful to determine cases where the equations $G_{ij}$ define the Rees algebra (Micali proved that a similar result holds over any commutative ring; see \cite[Th\'eor\`eme 1]{Mic64}).

\begin{lem}\label{eqReesCI}
Let $R$ be a Cohen-Macaulay local domain and $I=(f_1,\ldots ,f_r)$ {be} a complete intersection ideal of codimension $r$. Then the defining ideal $\IP$ of $\R_I$ is generated by the elements $G_{ij}:=f_iT_j-f_jT_i$.
\end{lem}

\begin{proof} 
Consider the $(2\times r)$-matrix $$M:=\begin{bmatrix} f_1&\cdots &f_r\\ T_1&\cdots &T_r\\ \end{bmatrix}.$$ 
Let $J:=I_2(M)=(G_{ij})$. The ideals $\IP$ and $J$ coincide off $V(I.S)$, which is of codimension $r$ in $S$. As $\IP$ is of codimension $r-1$, it follows that $J$ has depth at least $r-1$.  Hence the Eagon-Nortcott complex {associated to $M$} is a free $S$-resolution of $J:=(G_{ij})$. Thus $S/J$ is Cohen-Macaulay as well, hence unmixed. As $J$ and $\IP$ coincide off $V(I.S)$, it follows that $J=\IP$. 
\end{proof} 

\subsection{Symmetric algebra}\label{subsec:symalg} The elements in $\IP$ of $T$-degree $1$, that we can write $\IP _{*,1}$, correspond to linear forms $\sum_i a_iT_i$ with $a_i\in R$ such that $\sum_i a_i f_i=0$, that is to the first syzygies of the given generators of $I$, written as linear forms of the $T_i$'s with coefficients in $R$, in place of a $r$-tuple of elements in $R$. The surjection $S/(\IP_{*,1})\ra S/\IP $ is an incarnation of the canonical map $\sym_R (I)\ra \R_I$, whose kernel is the non linear part of $\IP$ (the torsion of the symmetric algebra). 

It is important to notice that, locally at a prime $\iq\in \Spec(R)$ where $I$ is a complete intersection, the symmetric and Rees algebras coincide, by Lemma \ref{eqReesCI}. This shows that the schemes in $X \times {\mathbb A}^{r}_k$ defined by $\sym_R (I)$ and $\R_I$ coincide whenever this holds for any $\iq \in \Proj_{G} (R/I)$. In other words, 
\begin{prop}\label{prop:graphci}
If $\Proj_{G} (R/I)$ is locally a complete intersection in $X$, then
$$
\R_I =\sym_R (I)/H^0_B (\sym_R (I)).
$$
\end{prop}
We notice that the case $\Proj_{G} (R/I)=\emptyset$, i.e.~$I$ contains a power of $B$, is contained in the above proposition.

\section{Elimination matrices and fibers of projections}\label{sec:elimmat}

In this section we assume that $X$ is a product of projective spaces and that the subscheme $\Gamma \subset X\times \iP_k^{r-1}$ is given by {finitely many} equations that are homogeneous with respect to the {$\ZZ^s$-}grading of the coordinate ring $R$ of $X$. We denote by $J \subset S=R[T_1,\ldots,T_r]$ the ideal generated by these equations, by $\Sc$ the quotient ring $R[T_1,\ldots,T_r]/J$ and by $B$ the irrelevant ideal associated to $X$. We also set $A:=k[T_1,\ldots,T_r]$.

\subsection{Elimination ideal}
Consider the canonical projection $\pi_2 : X\times \iP_k^{r-1} \rightarrow \iP_k^{r-1}$. It is a classical result that $\pi_2(\Gamma)$ is closed in $\iP_k^{r-1}$ and is defined by the elimination ideal
$$ \Af(J):=(J:B^\infty) \cap A.$$  
{An interesting fact is that the elimination ideal} can be connected to the graded components of the quotient ring $\Sc$. 

\begin{lem}\label{lem:annFitt} Assume that $\nu \in \NN^s$ is such that $H^0_B(\Sc)_\nu=0$, then $\Af(J)=\ann_A(\Sc_\nu)$. Moreover, there exists an integer $n_\nu$ such that
	$$ \Af(J)^{n_\nu} \subseteq \Fitt^0_A (\Sc_\nu) \subseteq  \Af(J)$$
where $\Fitt^0_A (\Sc_\nu)$ denotes the initial Fitting ideal of the $A$-module $\Sc_\nu$.
\end{lem} 
\begin{proof} $\Sc ':=\Sc /H^0_B(\Sc)$ is generated over $A/\Af(J)=(\Sc ')_0$ by the variables of $R$ and, as $B$ is contained in the ideal generated by any of {these} $s$ sets of variables of $R$ ($B$ is the intersection of the ideals $B_i$ generated by the sets of variables), $\Sc '$ is saturated with respect to $B_i$. In particular, for any $\nu\in \NN^s$ there exists a non zero divisor of degree $\nu$ on $\Sc'$, proving the claimed equality. The inclusions derive from any finite presentation of the $A$-module $\Sc_\nu$.	
\end{proof}

Now, for any $\nu\in \NN^s$ such that $H^0_B(\Sc)_\nu=0$, {denote by} $\MM_\nu$ a presentation matrix of the $A$-module $\Sc_\nu$ :
$$ A^b \xrightarrow{ \MM_\nu} A^a \rightarrow \Sc_\nu \rightarrow 0.$$
As a consequence of Lemma \ref{lem:annFitt}, for any point $p\in \iP_k^{r-1}$ the corank of $\MM_\nu(p)$ (the evaluation of $\MM_\nu$ at the point $p$) is positive if and only if $p\in \pi_2(\Gamma)$. This result can actually be refined as follows. 

\subsection{Finite fibers}
Let $p\in \iP_k^{r-1}$ and denote by $\Sc_p$ the specialization of $\Sc$ at $p$. From the definition of $\MM_\nu$ and the stability of Fitting ideals under arbitrary base change (which is not the case for annihilators), for any $\nu$,  the corank of $\MM_\nu(p)$ is nothing but the Hilbert function of $\Sc_p$ in degree $\nu$. Thus, if the fiber of $p$ under $\pi_2$ is finite then its Hilbert polynomial is a constant which is equal to the degree of this fiber. The next result shows that the degree $\nu$ such that the Hilbert function of $\Sc_p$ reaches its Hilbert polynomial is globally controlled by the vanishing of the local cohomology of $\Sc$.

\begin{thm}\label{specizero} Let $\nu\in \NN^s$ be such that $H^0_B(\Sc)_\nu=H^1_B(\Sc)_\nu=0$ and suppose given $p\in \iP_k^{r-1}$ such that the fiber $\pi_2^{-1}(p)\subset X$ is finite (a scheme of dimension zero or empty), then 
	$$\corank (\MM_\nu(p)) = \deg (\pi_2^{-1}(p)).$$
\end{thm}
\begin{proof}[Sketch of proof] First, Grothendieck-Serre's formula (see e.g.~\cite[Proposition 4.26]{BotCh}) shows that, for any $\nu\in \ZZ^s$, 
$$
\dim_k (\Sc_p)_\nu =\deg (\pi_2^{-1}(p))+\sum_{{i\geq 0}} (-1)^i\dim_k H^i_B(\Sc_p)_\nu ,
$$
and by Grothendieck vanishing theorem $H^i_B(\Sc_p)=0$ for $i>1$, as $\pi_2^{-1}(p)$ is of dimension zero or empty (in which case it also holds for $i=1$).

Let $\ip$ be the graded ideal defining {the point} $p$ and {let} $k_p:=A_\ip/\ip A_\ip$ {be} its residue field. {Since} $H^i_B(\Sc \otimes_A A_\ip)_\nu =H^i_B(\Sc )_\nu \otimes_A A_\ip$, we may replace $A$ by $A_\ip$, $S$ by $S\otimes_A A_\ip$ and $\Sc$ by $\Sc \otimes_A A_\ip$ to assume that $A$ is local with residue field $k_p$.

Next one uses a construction as in \cite[Lemma 6.2]{Cha12} to show the existence, for any finitely generated $S$-module $M$ and $A$-module $N$, of two spectral sequences with same abutment and second terms :
$$
{^{'}_2}E^p_q=H^p_B(\tor_q^A(M,N)) \quad {\rm and}\quad {^{''}_2}E^p_q=\tor_q^A(H^p_B(M),N).
$$
As the support in $S$ of $\tor_q^A(M,N)$ sits inside the one of $M\otimes_A N$, one deduces by choosing  $M:=\Sc$   and $N:=k_p$, first that $\max \{ i\ \vert H^i_B(\Sc )\not= 0\} =\max \{ i\ \vert H^i_B(\Sc \otimes_A k_p)\not= 0\}=1$, and then that $H^1_B(\Sc)_\nu=0\ \Leftrightarrow H^1_B(\Sc \otimes_A k_p)_\nu =0$ and $H^0_B(\Sc)_\nu=H^1_B(\Sc)_\nu =0\ \Rightarrow H^0_B(\Sc \otimes_A k_p)_\nu =0$.
\end{proof}

\section{When the source is $\iP^{1}$}\label{sec:P1}

In this section we focus on maps whose source is a projective line, which covers the case of rational curves embedded in a projective space of arbitrary dimension. We first describe how fibers of curve parameterizations can be obtained from elimination matrices, following the methods introduced in Section \ref{sec:elimmat}. Then, in the second part of this section we derive an upper bound for the Castelnuovo-Mumford regularity of rational curves by means of similar elimination techniques. The definition of Castelnuovo-Mumford regularity is also quickly reviewed.    

\subsection{Matrix representations}\label{subsec:mrepcurves}
{For simplicity, we deviate from general notations} and write $R:=k[x,y]$. We suppose given a rational map
\begin{eqnarray*}
	\psi: \iP^1_k & \rightarrow & \iP^{r-1}_k \\
	(x:y) & \mapsto & (f_1:\cdots:f_r) 	
\end{eqnarray*}
where the $f_i$'s are homogeneous polynomials in $R$ of degree $d\geq 1$ without common factor. Then, the image of $\psi$ is a curve $\Cc\subset \iP^{r-1}_k$ and $\psi$ is a morphism, in other words it has no base point. A classical intersection theory formula yields the equality
$$\deg(\psi)\deg(\Cc)=d.$$ 

In this setting, the equations of the symmetric algebra of $I=(f_1,\ldots,f_r)$ define the graph of $\psi$, because, as we assume that the $f_i$'s have no common factor, $\psi$ is a morphism. By Hilbert-Burch Theorem, these equations have the following nice structure: these exist non-negative integers {$\mu_1,\ldots,\mu_{r-1}$} such that {$\sum_{i=1}^{r-1} \mu_i=d$} and $R/I$ has a minimal finite free resolution of the form 
$$
\xymatrix{ 0\ar[r]&\sum_{i=1}^{r-1}R(-d-\mu_i )\ar^(.6){\Phi}[r] &R(-d)^r\ar[r]&I \ar[r]&0.}
$$
The columns of the map $\Phi$ yield a basis of the first syzygy module of $I$. It follows that the forms $L_1,\ldots,L_{r-1}$ defined as 
$$
{\begin{bmatrix}L_1\\ \vdots\\ L_{r-1}\\ \end{bmatrix}}= {^t\Phi} {\begin{bmatrix}T_1\\ \vdots\\ T_r\\ \end{bmatrix}}
$$
are defining equations of the symmetric algebra $\sym_R (I)$. Therefore, for any integer $\nu$ the multiplication map 
$$\oplus_{i=1}^{r-1} R_{\nu-\mu_i}[T_1,\ldots T_r](-1)^{r-1} \xrightarrow{(L_1,\ldots,L_{r-1})} R_{\nu}[T_1,\ldots T_r]$$
is a presentation of $\sym_R (I)_{\nu}$ by free {graded} $A$-modules (recall $A=k[T_1,\ldots,T_r]$). We denote by $\MM_\nu$ its matrix in some bases and set $B:=(x,y)$.

\begin{prop}\label{prop:curve} For all $\nu \geq \max_{i\not= j}\{ \mu_i +\mu_j\}$, $H^0_B(\sym_R (I))_\nu=H^1_B(\sym_R (I))_\nu=0$. Hence the matrix $\MM_\nu$ satisfies
	$$\corank(\MM_\nu(p)) =\deg(\psi^{-1}(p)),\quad \forall p\in \iP^{r-1}_k.$$
\end{prop}
\begin{proof} As $L_1,\ldots,L_{r-1}$ form a complete intersection in $\PP^1_k\times \PP^{r-1}_k$, the claimed vanishing of the local cohomology modules follows from the comparison of two spectral sequences associated to the \v{C}ech-Koszul double complex of $L_1,\ldots,L_{r-1}$ (see \cite[\S 2.10]{Jou80}). The property on the corank of the matrix $\MM_\nu(p)$ then follows from Theorem \ref{specizero}. 
\end{proof}

\begin{exmp}[Plane curves] In the case $r=3$, the map $\psi$ defines a curve $\Cc$ in the projective plane and $\mu_1+\mu_2=d$. In suitable bases, the matrix $\MM_{d-1}$ is nothing but the Sylvester matrix associated to $L_1$ and $L_2$ and we recover here classical results. In particular, its determinant is equal to $F^{\deg(\psi)}$, {where $F$ is an implicit equation of $\Cc$}. 
\end{exmp}

\begin{exmp}[Twisted cubic] Consider the following map whose image is the twisted cubic 
	\begin{eqnarray*}
		\psi : \iP^1_k & \rightarrow & \iP^3_k \\
		(x:y) & \mapsto & (x^3:x^2y:xy^2:y^3).
	\end{eqnarray*} 
In this case $\mu_1=\mu_2=\mu_3=1$ and we get the matrix
$$\MM_{1}=
\left(
\begin{array}{cccc}
	-T_2 & -T_3 & -T_4 \\
	T_1 & T_2 & T_3
\end{array}
\right).$$
Applying Proposition \ref{prop:curve}, we deduce that the twisted cubic is a smooth curve and $\psi$ is an isomorphism, because $\corank(\MM_1(p))\leq 1$ for any $p\in \iP^3_k$.
\end{exmp}

\subsection{Regularity estimate for rational curves}

As illustrated above, presentation matrices of the graded components of the symmetric algebra $\sym_R (I)$ provide an interesting computational tool to manipulate rational curves. Let us recall the notion of Castelnuovo-Mumford regularity, which controls degrees of generators of an ideal, and much more : degrees of all syzygies (as we will see) and degrees of generators of Groebner bases for reverse lexicographic order and general coordinates. 

\subsubsection{The four equivalent definitions of Castelnuovo-Mumford regularity.}

Assume $A$ is a Noetherian ring (hence $S=A[x_1,\ldots, x_n]$ as well) and $M$ is a finitely generated graded $S$-module. And assume further that $\deg (x_i)=1$ for all $i$ and $\deg (a)=0$ for $a\in A$ (standard grading). In what follows, {we will denote by $K_\bullet (f;N)$} the Koszul complex associated to a sequence of elements  $f$ on a graded module $N$. 

\medskip

As $K_\bullet (x_1,\ldots ,x_n;S)$ is a resolution of the $S$-module $A=S/(x_1,\ldots, x_n)$, by definition of $\Tor$ functors one has first
$$
\Tor_i^S (M,A)\simeq H_i(K_\bullet (x_1,\ldots ,x_n;M)).
$$
By Noetherianity, $\Tor_i^S (M,A)$ is a finitely generated graded $A$-module (hence not zero in only finitely many degrees). If $N$ is a graded module, we set
$$
\fin (N):=\sup \{ \mu\ \vert\ N_\mu \not= 0\}\in \ZZ\cup\{ \pm \infty\} ,
$$
write $S_+:=(x_1,\ldots ,x_n)$ and define
$$a^i(M):=\fin (H^i_{S_+}(M))\in \ZZ\cup\{ -\infty\}, \ \ b_j (M):=\fin (\Tor_j^S (M,R))\in \ZZ\cup\{ -\infty\}.$$
Recall that, as $M$ is finitely generated, it is generated in degrees at most $b_0(M)$, and this estimate is optimal. We are now ready to introduce regularity (see \cite[Proposition 2.4]{CJR} and \cite[\S 2]{Cha12}). 

\begin{thm}
Let $M\not= 0$ be a finitely generated graded $S$-module. Then,
\begin{multline*}
\reg (M):=\max_i \{ a^i(M)+i\}=\max_j \{ b_j(M)-j\}=\min_{F^M_\bullet}\{\sup_j \{ b_0 (F^M_j)-j\}\} \\ =\min_{F^M_\bullet}\{\max_{j\leq n}\{ b_0 (F^M_j)-j\}\}\in \ZZ ,	
\end{multline*}

where $F_\bullet^M$ runs over graded free $S$-resolutions of $M$. 
\end{thm}

Notice that $a^i(M)=-\infty$ unless $0\leq i\leq n$ and that $b_j(M)=-\infty$ unless $0\leq j\leq n$. However, $F^M_j$ could be non zero for any $j$, even if $F^M_\bullet$ is a minimal resolution over a local ring $A$, unless $A$ is local regular (in which case $F^M_j=0$ for $j>n+\dim A$,  if $F^M_\bullet$ is minimal).

\begin{rem}\label{remDefReg}

(1) If $A$ is local (or *local : graded with a unique graded maximal ideal) and $F_\bullet^M$ is a minimal graded free $S$-resolution of $M$, then 
$$
\reg (M)=\max_j \{ b_0 (F^M_j)-j\}=\max_{j\leq n}\{ b_0 (F^M_j)-j\} .
$$

(2) {If} $d:=\cohd_{S_+}(M):=\max\{ i\ \vert\ H^i_{S_+}(M)\not= 0\}$ {denotes} the cohomological dimension of $M$ relatively to $S_+$, then
$$
\reg (M)=\max_{ n-d\leq j\leq n} \{ b_j (M)-j\}.
$$
For instance, whenever $A$ is a field and $M$ is Cohen-Macaulay of dimension $d$, 
$$\reg (M)= b_{n-d}(M)-n+d.$$
\end{rem}

Recall that for a Noetherian local ring $(A,\im )$ and a finitely generated $A$-module $M$, $\cohd_\im (M)=\dim M :=\dim (A/\ann_A (M))$ : the cohomological dimension with respect to the maximal ideal is the dimension of the support of the module.

Let us briefly present the geometrical interpretation of the cohomological dimension with respect to $S_+$. For any prime ideal {$\ip\in \Spec (A)$}, set $k_\ip :=A_\ip /\ip A_\ip={\rm Frac}(A/\ip )$. The stalk of $\Fi :=\tilde M$ at $V(\ip )$ is the sheaf defined on $\iP^{n-1}_{\Spec (A_\ip )}$ by $M\otimes_A A_\ip$. The fiber of $\Fi$ at the corresponding point is the sheaf defined on $\iP^{n-1}_{k_\ip }$ by $M\otimes_A k_\ip$, with the usual abuse of notations $\iP^{n-1}_{k_\ip }:=\iP^{n-1}_{\Spec (k_\ip )}$. The following result shows in particular that, in our situation, the cohomological dimension is the maximal dimension of the fibers of the family of sheaves given by $\Fi$ (plus one, as there is a difference of 1 between the dimension of a graded module and the one of the sheaf on the projective sheaf it represents).

\begin{lem}[{\cite[Proposition 6.3]{Cha12}}]
(1) For any $\ip \in \Spec (A)$, $$H^i_{S_+}(M)\otimes_A A_\ip \simeq H^i_{S_+}(M\otimes_A A_\ip ).$$

(2) If $(A,\im ,k)$ is local, then
$$
d:=\max_{\ip \in \Spec (A)}\{ \dim (M\otimes_A A_\ip/\ip A_\ip )\} =\max\{ i\ \vert\ H^i_{S_+} (M\otimes_A k)\not= 0\}=\max\{ i\ \vert\ H^i_{S_+} (M)\not= 0\},
$$
and 
$H^d_{S_+}(M)\otimes_A k\simeq H^d_{S_+} (M\otimes_A k).$
\end{lem}

Hence, by (1) and (2), $\cohd_{S_+}(M)$ is the maximal dimension over the prime ideals $\ip$ of $A$ (equivalently among the maximal ideals) of the modules $M\otimes_A A_\ip /\ip A_\ip$. In other words, it is one more than the maximal dimension of support of the fibers of the family of sheaves $\Fi$ over {$\Spec(A)$}. 

The notion of regularity extends to Cox rings, taking cohomology with respect to $B$ or more generally with respect to any graded $S$-ideal; we refer the reader to \cite{MlS04} and \cite{BotCh}.

\subsubsection{The regularity estimate for rational curves} For the case of rational curves in a projective space, the next result gives another application of the matrices introduced in Section \ref{subsec:mrepcurves}, for Castelnuovo-Mumford regularity estimation. The proof is very much related to the original argument of L'vovsky \cite{Lvo96}, which itself relies on work of Gruson, Lazarsfeld and Peskine \cite{GLP83}; our presentation has a more ring theoretic flavor.

\begin{thm}\label{reg2}
Let $\psi :\iP^{1}_k \lra \iP^{r-1}_k$ be a morphism defined by $(f_1,\ldots ,f_r)$ and $I_\Cc$ be the defining ideal of the image of $\psi$.
If $\psi$ is birational to its image, then {$\sum_{i=1}^{r-1} \mu_i =d$} and
$$
\reg (I_\Cc )\leq \max_{i\not= j}\{ \mu_i +\mu_j\}.
$$
In particular, $\reg (S/I_\Cc )\leq \deg (\Cc ) -\codim (\Cc )$ if the $f_i$'s are linearly independent (in other words, {if} the curve $\Cc$ is not sitting in any hyperplane).
\end{thm}

\begin{proof}
Let $\rho :=\max_{i\not= j}\{ \mu_i +\mu_j\}$. 
From Proposition \ref{prop:curve}, it follows that $H^0_{(x,y)} (\Ss_I)_\mu =0$ for $\mu\geq \rho -1$, in particular $(\Ss_I)_\mu =(\R_I)_\mu$ for $\mu\geq \rho -1$. This provides a presentation for the $A$-module $(\R_I)_{\rho -1}$ :
$$
\xymatrix@C=25pt{
\oplus_i A[x,y]_{\rho -1-\mu_i}\ar^(.52){\begin{bmatrix} &\vdots&\\ \cdots &L_{ij}& \cdots\\ &\vdots&\\ \end{bmatrix}}[rr]&&A[x,y]_{\rho -1}\ar[r]&(\R_I)_{\rho -1}\ar[r]&0\\}
$$
where, for $j$ corresponding to $x^ay^b\in A[x,y]_{\rho -1-\mu_i}$ one has $x^ay^b L_{i}=\sum_{i=1}^{\rho }x^{i-1}y^{\rho -i}L_{ij}(T_1,\ldots ,T_r)$. 
We now use the graded presentation of the $A$-module $(\R_I)_{\rho -1}$ as above 
$$
\xymatrix{ A[-1]^N \ar^(.6){\Psi}[r]&A^\rho \ar[r]&(\R_I)_{\rho -1}\ar[r]&0\\}
$$
where $\Psi $ is the matrix of the linear forms $L_{ij}$. The ideal $J:=\Fitt^0_A ((\R_I)_{\rho -1})$ is {an ideal of $A$} generated in degree $\rho$ (by the maximal minors of $\Psi$). Let $\inn :=(T_1,\ldots ,T_r)$ be the graded {irrelevant} ideal of $A$. The conclusion will derive easily from the two following observations :
(1) $J=I_\Cc \cap \ia$ with $V(\ia )$ supported on the finitely many points of $\Cc$ where $\psi$ is not locally invertible (and thus an isomorphism), (2) $\reg (J^{sat} )\leq \rho$.

Indeed, $J^{sat}=I_\Cc \cap \ia '$ with $V(\ia )=V(\ia ')$, and the exact sequence $$0\ra I_\Cc / J^{sat} \ra S/J^{sat} \ra S/I_\Cc\ra 0$$ provides an exact sequence 
$$
\xymatrix{ H^1_\inn (S/J^{sat})\ar[r]& H^1_\inn (S/I_\Cc )\ar[r]&H^2_\inn ( I_\Cc / J^{sat})=0\\}
$$
and an isomorphism $H^2_\inn (S/I_\Cc )\simeq H^2_\inn (S/J^{sat})$, proving that $\reg (I_\Cc )\leq \reg (J^{sat} )$.

To prove (1) notice that if $z\in \Cc$ is such that $\pi :=\Gamma \ra \Cc$ is {locally an isomorphism}  at $z=V(\ip )$, then $\R_I \otimes_A A_\ip\simeq (A/I_\Cc )_\ip [x]$ - in other words, the defining ideal of $\R_I$ over $(A/I_\Cc ) [x,y]$ contains an equation $fx+gy$ with $f,g\in A/I_\Cc $ and $g\not\in \ip$. Hence, for any $\mu \geq 0$, $(\R_I \otimes_A A_\ip )_\mu \simeq A_\ip /(I_\Cc )_\ip$ admits a presentation 
$$
\xymatrix{ A_\ip ^t \ar^{\Theta}[r]&A_\ip \ar[r]&(\R_I \otimes_A A_\ip )_\mu\ar[r]&0\\},
$$
where the entries of $\Theta$ are generators of $I_\Cc\otimes_A A_\ip$. It follows that $J\otimes_A A_\ip =I_\Cc\otimes_A A_\ip$, as Fitting ideals are independent of the presentation. This proves (1).

We now prove (2). The Eagon-Northcott complex $E_\bullet$ of the matrix $\Psi$ has the form :
$$
\xymatrix{ \cdots\ar[r]&A[-\rho -2]^{a_3}\ar[r]& A[-\rho -1]^{a_2}\ar[r]&A[-\rho ]^{a_1}\ar^(.6){\bigwedge^\rho \Psi}[r]\ar[r]&A\ar[r]&0\\}.
$$
We consider the double complex $\Cc^\bullet_{\inn}(E_\bullet)$ to estimate the regularity of $H_0(E_\bullet )=A/J$, {where the notation $\Cc^\bullet_{\inn}(-)$ stands for the \v{C}ech complex with respect to the ideal $\inn$}. 
Setting $H_i:=H_i(E_\bullet )$, the two spectral sequences have respective second pages :
$$
\xymatrix{ 
\cdots &H^0_{\inn}(H_{2})&H^0_{\inn}(H_{1})\ar[ddl]&H^0_{\inn}(A/J)\ar[ddl]\\
\cdots &H^1_{\inn}(H_{2})&H^1_{\inn}(H_{1})&H^1_{\inn}(A/J)\\
\cdots &H^2_{\inn}(H_{2})&H^2_{\inn}(H_{1})&H^2_{\inn}(A/J)\\
\cdots & 0 &0 &0\\}
\quad\quad  \xymatrix{ \ \\{\rm and}\\ \ \\}\quad\quad
\xymatrix{ 0&\cdots &0\\
0&\cdots &0\\
\cdots &H_i(H^n_{\inn} (E_\bullet))&\cdots \\}
$$
As a consequence, $H^2_{\inn}(A/J)\simeq H_{n-2}(H^n_{\inn} (E_\bullet))$ and $H_{n-1}(H^n_{\inn} (E_\bullet))_\mu =0$ {implies that} $H^1_{\inn}(A/J)_\mu =0$.
But $H^n_{\inn} (E_{p})\simeq H^n_{\inn}(A(-\rho -p+1)^{a_{p}})$ for $p\geq 1$ vanishes in degrees $>\rho +p-1-n$. 
This shows that $H^1_{\inn}(A/J)_\mu =0$ for $\mu >\rho -2$ and $H^2_{\inn}(A/J)_\mu =0$ for $\mu >\rho -3$. Hence $\reg (S/J^{sat})\leq \rho -1$, as claimed.
\end{proof}

\section{Morphisms from $\iP^{n-1}_k$ to $\iP^{n}_k$}\label{sec:Pn-1}

A morphism has no base point and hence only has finite fibers. In this section, we will consider morphisms associated to hypersurface parameterizations and 
present some results that partially extends to more general situations, in particular to rational maps whose base locus is of dimension zero (see \cite[Section 4]{BCS10}).
We keep notations as in Section \ref{sec:graph}; we consider a morphism
\begin{eqnarray*}
\psi: \PP^{n-1}_k & \rightarrow & \iP^{n}_k \\ \nonumber
x=(x_1:\cdots:x_n) & \mapsto & (f_1 (x):\cdots :f_{n+1}(x)),
\end{eqnarray*}
and set $R:=k[x_1,\ldots,x_n]$, $A:=k[T_1,\ldots, T_{n+1}]$. 

\medskip

The original approach of Jouanolou in this situation was to notice first that the Rees algebra  is the saturation of the symmetric algebra, {i.e.}~$\R_I =\sym_R(I) /H^0_\im (\sym_R(I) )$ with $\im =(x_1,\ldots ,x_n)$, and second that $\sym_R(I)$ admits a resolution by the approximation complex {of cycles}, whose graded components provide free $A$-resolutions of $\sym_R(I)_{\mu ,*}$. Hence, given $\mu \geq 0$ such that $H^0_\im (\sym_R(I) )_{\mu ,*}=0$, one gets a (minimal) free $A$-resolution of the $A$-module $(\R_I )_{\mu ,*}$, whose annihilator is the ideal of the image.\\

We briefly recall what the approximation complex ({of} cycles) is. The two Koszul complexes
$K_{\bullet}(f;S)$ and $K_{\bullet}(T;S)$ where $f:=(f_{1},\ldots
,f_{n+1})$ and $T:=(T_{1},\ldots
,T_{n+1})$, have the same modules
$K_{p}=\bigwedge^{p}S^{n+1}\simeq S^{{n+1}\choose{p}}$ and 
differentials $d^{f}_{\bullet}$ and $d^{T}_{\bullet}$, {respectively}. Set $Z_{p}(f;S):=\ker (d^{f}_{p})$. 
It directly follows from the definitions that
$d^{f}_{p-1}\circ d^{T}_{p}+d^{T}_{p-1}\circ d^{f}_{p}=0$, so that 
$d^{T}_{p}(Z_{p}(f;S))\subset Z_{p-1}(f;S)$. 
The complex
$\Z_{\bullet}:=(Z_{\bullet}(f;S),d^{T}_{\bullet})$ is the 
$\Z$-complex associated to the $f_{i}$'s.
Notice that $Z_{p}(f;S)=S\otimes_{R}Z_{p}(f;R)$, $Z_{0}(f;R)=R$, $Z_{1}(f;R)={\rm Syz}_{R}(f_{1},\ldots ,f_{n+1})$, and the map $d_{1}^{T}$ is defined by 
\begin{eqnarray*}
	d_{1}^{T}:S\otimes_{R}{\rm Syz}_{R}(f_{1},\ldots
	,f_{n+1}) & \rightarrow & S \\
	(a_{1},\ldots ,a_{n+1}) & \mapsto & a_{1}T_{1}+\cdots +a_{n+1}T_{n+1}.
\end{eqnarray*}

The following result, that holds for any finitely generated ideal $I$ in a commutative ring $R$, shows the intrinsic nature of the homology of the
$\Z$-complex, it is a key point in proving results on its
acyclicity (see \cite[Section 4]{HSV}).\medskip 

\begin{thm} $H_{0}(\Z_{\bullet})\simeq \sym_R(I)$ and the
homology modules $H_{i}(\Z_{\bullet})$ are $\sym_R(I)$-modules that only
depend upon $I\subset R$, up to isomorphism.
\end{thm}

Now, we consider graded pieces :
$$
\Z_{\bullet}^{\nu}\ :\ \cdots\lra 
A\otimes_{k}Z_{2}(f;R)_{\nu +2d}\buildrel{d_{2}^{T}}\over{\lra}
A\otimes_{k}Z_{1}(f;R)_{\nu +d}\buildrel{d_{1}^{T}}\over{\lra} 
A\otimes_{k}Z_{0}(f;R)_{\nu}\lra 0
$$ 
where $Z_{p}(f;R)_{\nu +pd}$ is the part of  $Z_{p}(f;R)$ consisting of
elements of the form $\sum a_{i_{1}\cdots i_{p}}e_{i_{1}}\wedge\cdots
\wedge e_{i_{p}}$ with the $a_{i_{1}\cdots i_{p}}$'s all of the same
degree $\nu$. Bus\'e and Jouanolou proved in \cite[Theorem 5.2 and Proposition 5.5]{BuJo03} that (we refer the reader to \cite[{Appendix A}]{GKZ94} for the notion of determinant of complexes)~:

\begin{prop}\label{mfr1} 
$\Z_{\bullet}$ is acyclic and for any $\nu \geq (n-1)(d-1)$, $H^0_\im (\sym_R(I) )_{\nu ,*}=0$. Hence, for such {an integer} $\nu$, $\Z_{\bullet}^{{\nu}}$ is a minimal free $A$-resolution of $(\R_I)_{{\nu}}$ and $\det (\Z_{\bullet}^{\nu})=H^{\delta}$, where $H$ is the equation of the image and $\delta$ the degree of the map $\psi$ onto its image.
\end{prop}

In addition, it can be seen that for these degrees $\nu$, the fibers of $\psi$ can be obtained by means of elimination matrices. More precisely, let $\MM_\nu$ be a matrix of the first map of the complex $\Z_{\bullet}^{\nu}$, i.e.~the map of free $A$-modules
$$\Z_{1}^{\nu}=A\otimes_{k}Z_{1}(f;R)_{\nu +d} \xrightarrow{d_{1}^{T}} \Z_{0}^{\nu}=A\otimes_{k}R_{\nu}.$$

\begin{prop}\label{prop:MRepmorph} Let $\nu\geq (n-1)(d-1)$ and suppose given $p\in \PP^{n}_k$, then
	$$\corank(\MM_\nu(p))=\deg(\psi^{-1}(p)).$$
\end{prop}
\begin{proof} As the matrix $\MM_\nu$ is a presentation matrix of the $A$-module $\sym_R(I)_{\nu ,*}$, according to Theorem \ref{specizero} one has to prove that $H^i_\im (\sym_R(I) )_{\nu ,*}=0$ for $i=0,1$ and for all $\nu\geq (n-1)(d-1)$. This follows from Proposition \ref{mfr1} for $i=0$ and the case $i=1$ can be proved in the same vain; we refer the reader to the proof of \cite[Proposition 5]{BBC14}. 
\end{proof}

A useful structural result to go further, in particular to get more compact elimination matrices, is the following. Recall that $\IP$ denotes the defining ideal of the Rees algebra $\R_I$ in $S=R[T_1,\ldots ,T_{n+1}]$. We define $\IP \lr{\ell}$ as the ideal generated by equations of the Rees algebra of $T$-degree at most $\ell$. In particular $\sym_R(I) =S/\IP\lr{1}$ and $\R_I =S/\IP\lr{\ell}$ for $\ell\gg 0$.

\begin{prop}[{\cite[Corollary 1]{BCS10}}]\label{mfr2} For every  $\nu \geq \nu_0 (I):=\reg (I)-d$,
			the  $A$-module $(\R_I)_{\nu ,*}$ admits a minimal graded free $A$-resolution of the form
			\begin{equation}\label{eq:resolutionSI*}
				\cdots \ra \Z_i^\nu \ra \cdots \ra \Z_2^\nu \ra \Z_1^\nu
				\oplus_{\ell=2}^{n}({\IP}\lr{\ell}/{\IP}\lr{\ell -1})_{\nu} \ra
				\Z_0^\nu
			\end{equation}
			with $\Z_i^\nu =A\otimes_{k}Z_{i}(f;R)_{\nu +id}$.
		\end{prop}

Notice that $\Z_1^\nu =({\IP}\lr{1}/{\IP}\lr{0})_{\nu}$. In comparison with the {admissible} degrees in Proposition \ref{mfr1}, the gain is potentially quite big in terms of the range of degrees it concerns due to the following : 

\begin{lem}[{\cite[Corollary 3]{BCS10}}]\label{mu0gen} The threshold degree $\nu_0(I)$ defined in Proposition \ref{mfr2} satisfies the inequalities
	$$\left\lfloor {{(n-1)(d-1)}\over{2}}\right\rfloor \leq \nu_0(I) \leq (n-1)(d-1).$$
{Moreover, the} equality on the left holds if the forms $f_1,\ldots ,f_{n+1}$ are sufficiently general
 and $k$ has characteristic $0$.
\end{lem}

Above the threshold degree $\nu_0 (I):=\reg (I)-d=\reg (R/I)+1-d$, the following nice property holds :
		\begin{equation}\label{eq:JlJl-1=H1}
			({\IP}\lr{\ell}/{\IP}\lr{\ell -1})_{\nu }\simeq (H_1)_{\nu +\ell d}\otimes_k A[-\ell\ ],\quad \forall \nu \geq \nu_0 (I),\quad \forall \ell >1,
		\end{equation}
where the brackets stands for degree shifting in the $T_i$'s and $H_1$ denotes the homology module of the Koszul complex $K_\bullet(f;R)$. It turns out that the isomorphism \eqref{eq:JlJl-1=H1} is very explicit : given a non-Koszul syzygy $s:=(h_1,\ldots ,h_{n+1})$ with $\deg (h_i)=\nu +(\ell -1)d$, which corresponds to an element in 
$(H_1)_{\nu +\ell d}$, or equivalently to the class of $h_1T_1+\cdots +h_{n+1}T_{n+1}$ in $\IP\lr{1}/{\rm im}(d_2^T)$, one can write $h_i=\sum_{\vert \alpha\vert = \ell -1} c_{i,\alpha }f^{\alpha}$ 
as 
$$
\deg (h_i) \geq \nu_0(I) +(\ell -1)d =\reg (R/I)+1+(\ell -2)d\geq \reg (R/I^{\ell -1})+1= \mathrm{end}(R/I^{l-1}) +1
$$
(see \cite[Theorem 1.7.1]{Cha07} for the last above inequality).
The map sends $s$ to the element $\sum_i\sum_\alpha c_{i,\alpha}T_iT^\alpha$. In the other direction, one writes (the class) of an element of $({\IP}\lr{\ell}/{\IP}\lr{\ell -1})_{\nu }$ in the form $\sum_i\sum_\alpha b_{i,\alpha}T_iT^\alpha$ and maps it to $(\sum_\alpha b_{1,\alpha}f^\alpha ,\ldots ,\sum_\alpha b_{n+1,\alpha}f^\alpha )$. After checking that both maps are well-defined, it is clear that one is the inverse of the other. They are called upgrading and downgrading maps (which refers to the degree in the $T_i$'s of the elements).

One can interpret \eqref{eq:JlJl-1=H1} as a description of the graded strands of $\R_I$ in degrees at least the threshold degree $\nu_0 (I)$ purely in terms of syzygies : the non-linear equations are all obtained by upgrading some non-Koszul syzygies, and this is a one-to-one correspondence. For the cases of small dimension, it was already observed previously that incorporating quadratic relations allowed to work in degrees smaller than $(n-1)(d-1)$, hence provides matrix representations of smaller size (but with quadratic, or linear and quadratic entries); see e.g.~\cite{CGZ00,Cox01,BCD03}. In the general case, the following result holds.

\begin{thm}\label{mrdz} For every  $\nu \geq \nu_0 (I)$, a matrix $\MM_\nu$ of the $A$-linear map (extracted from \eqref{eq:resolutionSI*})
			\begin{equation}
				\quad \oplus_{\ell=1}^{n}({\IP}\lr{\ell}/{\IP}\lr{\ell -1})_{\nu} \lra
				R_\nu [T]=A^{{\nu +n-1}\choose{n-1}}
			\end{equation}
			is a \emph{matrix representation} of the fibers of $\psi$ : for any $p\in \PP^{n}_k$, $\corank(\MM_\nu(p))=\deg(\psi^{-1}(p))$. Furthermore, for $\ell >1$,  $k$-bases of $({\IP}\lr{\ell}/{\IP}\lr{\ell -1})_{\nu ,\ell}$  and $(H_1)_{\nu +\ell d}$ are in one-to-one correspondence via the upgrading and downgrading maps described above.
		\end{thm}

\begin{proof}
The algebra $S^{(\nu )}:=S/(\IP_{*,1}+\IP_{\geq \nu,*})=\sym_R(I)/H^0_\im (\sym_R(I))_{\geq \nu ,*}$ satisfies the equalities 
 $H^0_\im (S^{(\nu )})_{\geq \nu ,*}=0$ and 
 $H^i_\im (S^{(\nu )})=H^i_\im (\sym_R(I))$ for $i>0$.
As $H^i_\im (\sym_R(I))_{\geq \nu ,*}=0$ for $i>0$ by \cite[Theorem 1]{BCS10}, the conclusion follows from \cite[6.3]{Cha12} since $H^i_\im (S^{(\nu )})_{\geq \nu ,*}=0$ for all $i$.
\end{proof}

\begin{rem}
In the particular case of a general map from $\iP^2_k$ to $\iP^3_k$ ($n=3$), Theorem \ref{mrdz} provides a square matrix of quadrics which is a matrix representation of the fibers : the determinant is the equation of the image and ideals of minors of different sizes yield a filtration of fibers by their degrees; see \cite[\S 5]{BCS10} for more details.	
\end{rem}

\section{When the source is of dimension 2}\label{sec:dim2}

In this section we consider the case of {a rational map $\psi :\iP^{2}_k \dasharrow \iP^{3}_k$}  defined by four homogeneous polynomials $f_1,\ldots,f_4$ of the same degree $d\geq 1$. Following the discussion in Section \ref{subsec:symalg}, the geometric picture is reflected by the inclusions $\Gamma \subseteq W\subset  \iP^{2}_k \times_k \iP^{3}_k$, with $\Gamma =\Proj(\R_I)$ and $W=\Proj(\Sc_I )$. We are seeking informations on both the image and the fibers of the second canonical projection $\pi$ from $\Gamma$ to $\iP^{3}_k$.

Following Proposition \ref{prop:graphci}, whenever $\Proj (R/I)\subseteq \iP^{2}_k$ is empty or locally a complete intersection (for instance reduced, as it is of dimension zero) then $\Gamma =W$. Actually, components of $W$ can be easily described~: besides $\Gamma$, they only differ by linear subspaces of the form $\{p\}\times L_p$ where $p$ is a point where $V(I)$ is not locally a complete intersection. Two cases happen~:
\begin{itemize}
	\item $V(I)$ is locally defined at $p$ by the four equations, and not less, in which case $L_p=\iP^{3}_k$ and therefore $\pi (W)=\iP^{3}_k$,
	\item $V(I)$  is locally defined at $p$ by three equations, in which case $L_p$ is a hyperplane whose equation is given by the specialization at $p$ of any syzygy that do not specialize to 0, equivalently the specialization of a local relation expressing one $f_i$ in terms of the other three. It can be verified (see \cite[Lemma 6]{BCJ06}) that this component has multiplicity equal to the difference between the Hilbert-Samuel local multiplicity of the ideal and its colength (these being equal if and only if $I$ is locally a complete intersection).
\end{itemize}

We keep notations as above, in particular  $I=(f_1,\ldots ,f_4)$ is a graded {ideal of $R$} generated in degree $d$.  Recall that $a^i(M):=\fin (H^i_\im (M))$ and $a^* (M):=\max_i \{ a^i (M)\}$, with $\im =(x_1,x_2,x_3)$. Finally, in a local ring the notation {$\mu (N)$} denotes the minimal number of generators of the module $N$.

In view of applying Theorem \ref{specizero}, a crucial step is to estimate the vanishing degree of local cohomology with respect to {$\im$}, which is reflected in the regularity of $\Sc_I$ as a $\ZZ$-graded $A[x_1,x_2,x_3]$-module ($\deg (x_i)=1$). 

\begin{prop}[{\cite[Proposition 5]{BBC14}}]\label{prop:regSym} Assume that $\dim(R/I)\leq 1$ and {$\mu(I_\ip)\leq 3$} for every prime ideal $\im\supsetneq \ip \supset I$, then $a^*(\Sc_I)+1\leq\nu_0:=2(d-1)-\indeg(I^{sat})$, and
$$
\reg(\Sc_I)\leq \nu_0,
$$ 
unless $I$ is a complete intersection of two forms of degree $d$.
\end{prop}

Notice that if $I$ is a complete intersection of two forms, it defines a surjective map from $\iP^{2}_k$ to $\iP^{1}_k$ that has no finite fiber. We will be interested by the case where $\psi$ is generically finite. To understand and compute fibers, Fitting ideals plays a key role. 

\subsection{Fitting ideals associated to $\psi$}
From the properties of matrix representations of $\psi$, we see that for all $\nu\geq \nu_0$ the Fitting ideal $\Fitt^0_A((\Sc_I)_\nu)$ is supported on $W$ and hence provides a scheme structure on $W$. Following \cite{Tei77} and \cite[V.1.3]{EH00}, it is called the \emph{Fitting image} of $\Sc_I$ by $\pi$.

\begin{rem}
	Observe that by definition, the ideals $\Fitt_{A}^0((\Sc_I)_\nu)$ depend on the integer $\nu$ (it is generated in degree $\nu$) whereas $\ann_A((\Sc_I)_\nu)=I_Z$, with $Z:=\pi (W)$ defined by the elimination ideal, for all $\nu\geq \nu_0$.
	\end{rem}

We can push further the study of $\Fitt_{A}^0((\Sc_I)_\nu)$ by looking at the other Fitting ideals $\Fitt_{A}^i((\Sc_I)_\nu)$, $i>0$, since they provide a natural stratification~:
$$ \Fitt_{A}^0((\Sc_I)_\nu) \subset \Fitt_{A}^1((\Sc_I)_\nu) \subset \Fitt_{A}^2((\Sc_I)_\nu) \subset \cdots \subset  \Fitt_{A}^{\binom{\nu+2}{2}}((\Sc_I)_\nu)=A.$$
These Fitting ideals are actually closely related to the geometric properties of the parameterization $\psi$.  For simplicity, the Fitting ideals $\Fitt_{A}^i((\Sc_I)_\nu)$ will be denoted $\Fitt^i_\nu(\psi)$. We recall that  $\Fitt^i_\nu(\psi) \subset A$ is generated by all the minors of size $\binom{\nu+2}{2}-i$ of any $A$-presentation matrix of $(\Sc_I)_\nu$.

\begin{exmp} Consider the following parameterization of the sphere 
\begin{eqnarray*}
	\psi : \PP^2_\CC & \dto & \PP^3_\CC \\
	 (x_1:x_2:x_3) & \mapsto & (x_1^2+x_2^2+x_3^2: 2x_1x_3:2x_1x_2:x_1^2-x_2^2-x_3^2).
\end{eqnarray*}	
It has two base points. Following Theorem \ref{specizero} and Proposition \ref{prop:regSym}, matrix representations $\MM_\nu$ have the expected properties for all $\nu\geq 2-1=1$. The computation of $\MM_1$ yields 
$${\MM_1}=
\left(\begin{array}{cccc}
 0 &   T_2 &      T_3 &     -T_1+T_4 \\
	T_2 & 0   &     -T_1-T_4 & T_3 \\
	-T_3 &  -T_1-T_4 & 0       & T_2  
\end{array}\right).$$
A primary decomposition of the $3\times 3$ minors of $\MM_1$, i.e.~$\Fitt^0_1(\psi)$, gives 
$$(T_1^2-T_2^2-T_3^2-T_4^2)\cap(T_3,T_2,T_1^2+2T_1T_4+T_4^2),$$
which corresponds to the implicit equation of the sphere plus one embedded double point $(1:0:0:-1)$. Now, a primary decomposition of the $2\times 2$ minors of $\MM_1$, i.e.~$\Fitt^1_1(\psi)$, is given by 
$$(T_3,T_2,T_1+T_4)\cap (T_4,T_3^2,T_2T_3,T_1T_3,T_2^2,T_1T_2,T_1^2),$$
which corresponds to the same embedded point $(1:0:0:-1)$ plus an additional component supported at the origin. Finally, the ideal of $1$-minors of {$\MM_1$}, i.e.~$\Fitt^2_1(\psi)$, is supported at $V(T_1,\ldots,T_4)$ and hence is empty as a subscheme of $\PP^3_k$.

The point $(1:0:0:-1)$ is actually a singular point of the parameterization $\psi$ (but not of the sphere itself). Indeed, the line $L=(0:x_2:x_3)$ is a $\PP^1$ that is mapped to {the point} $(x_2^2+x_3^2: 0:0:-(x_2^2+x_3^2))$. In particular, the base points of $\psi$, namely $(0:1:i)$ and $(0:1:-i)$, are lying on this line, and the rest of the points are mapped to $(1:0:0:-1)$. Outside $L$ at the source and $P$ at the target, $\psi$ is an isomorphism.
\end{exmp}

Applying Theorem \ref{specizero} with the estimate  of Proposition \ref{prop:regSym} gives the following result that explains the phenomenon just noticed in the example :

\begin{thm} Let $\pi : W\ra \iP^{3}_k$ be the second canonical projection of $W\subset \PP^2_k\times \PP^3_k$ and $p \in \iP^{3}_k$.
	If $\dim \pi^{-1}(p) \leq 0$ then, for all $\nu\geq \nu_0 :=2(d-1)-\indeg (I^{sat})$,
	\begin{equation}\label{eqFittDeg}
	p \in V(\Fitt^i_\nu(\psi)) \Leftrightarrow \deg(\pi^{-1}(p))\geq i+1.
	\end{equation}
\end{thm}

In other words, for any $\nu\geq \nu_0$, the Fitting ideals of a matrix representations stratify the fibers of dimension zero  (or empty) by their degrees. One could also remark that the existence of base points improves the value of $\nu_0$ ($\indeg (I^{sat}) =0$ when $V(I)=\emptyset$); furthermore, it was proved in \cite[Proposition 2]{BCJ06} that the value of $\nu_0$ is sharp in some sense.

\subsection{One dimensional fibers} When $p$ is such that the fiber has dimension one, then Theorem \ref{specizero} does not apply and in fact fails without further assumptions. In the special case we are {considering}, one can nevertheless obtain good estimates for the regularity. One way to see this uses the fact that such a fiber corresponds to a special form of the ideal, namely :

\begin{lem}[{\cite[Lemma 10]{BBC14}}]\label{1fiber}
Assume that the $f_i$'s  are linearly independent, the fiber over $p :=({p_1:p_2:p_3:p_4})\in \PP^3_k$ is of dimension 1, and its
unmixed component is defined by $h_p\in R$. 
Let $\ell_p$ be a linear form with 
$\ell_p (p )=1$ and set  $\ell_i (T_1,\ldots,T_4):=T_i-p_i \ell_p (T_1,\ldots,T_4)$.  
Then, 
$h_p =\gcd ({\ell_1(f),\ldots ,\ell_4(f)})$ 
and 
$$
I=(\ell_p (f))+h_p ({g_1,\ldots ,g_4})
$$ 
with $\ell_i (f)=h_p g_i$ and   $\ell_p (g_1,\ldots ,g_4 )=0$.
In particular 
$$
(\ell_p (f))+h_p (g_1,\ldots ,g_4 )^{sat}\subseteq I^{sat} \subseteq (\ell_p (f))+(h_p ).
$$
\end{lem}

In \cite[Theorem 12]{BBC14} an estimate on the regularity of the specialization of the symmetric algebra and of its Hilbert function (the one of the fiber) is derived, which is at least one less than the one given for fibers of dimension zero. However, this {does not provide} a very easy way to determine or control the fibers of dimension one. 

We now turn to this question using Lemma \ref{1fiber} and Jacobian matrices. Let $Z$ be the finite set of points $p\in \iP^{3}_k$ having a 1-dimensional fiber. Notice that the unmixed part ({i.e.~purely} 1-dimensional part) of a fiber of $\pi : W\ra \iP^{3}_k$ or $\pi' : \Gamma \ra \iP^{3}_k$ are equal, as the fibers may only differ at points where $V(I)$ is not locally a complete intersection.

Choose a linear form $\ell =\lambda_1 T_1+\cdots +\lambda_4 T_4$, with $\lambda_i\in k$, not vanishing at any point of $Z$ (i. e. a plane that does not meet $Z$) with non zero first coefficient (a general form for instance) and set $f:=\lambda_1 f_1+\cdots +\lambda_4 f_4$. Then for any point $p\in Z$ with a fiber whose unmixed part is {defined} by $h_p\in R$, there exists $g_{p,1},g_{p,2},g_{p,3}$ such that~:
$$
I=(f)+(h_p)(g_{p,1},g_{p,2},g_{p,3}).
$$
Examples {show} that for $4$ forms of degree $d$, one can find birational morphisms with a least $2d-2$ distinct 1-dimensional fibers, hence at least this number of distinct decompositions of this type (private communication of M.~Chardin  and Hoa Tran Quang). On the other hand, the following result gives an upper bound and a way to determine the 1-dimensional fibers :

\begin{thm}[{\cite[4.4]{CCT21}}]\label{thm:jmat} Let $J(f)$ be the Jacobian matrix of the $f_i$'s. Then the ideal $I_3(J(f))$ of maximal minors of this $3\times 4$ matrix is generated by 4 forms of degree $3(d-1)$. If not all zero, the {GCD} $F$  of these 4 forms is divisible by $\prod_{z\in Z}h_z$, hence :
$$
\sum_{z\in Z} \deg (h_z)\leq \deg F\leq 3(d-1)-\indeg ({\rm Syz}(I))
$$
(the degree of a homogeneous syzygy $\sum a_if_i =0$ is $\deg (a_i)$ for any $i$).
\end{thm}

We notice that if the characteristic is zero, the Jacobian ideal is not zero, and this also holds if the characteristic of $k$ doesn't divide $d$ and $k(X)$ is separable over $k(f)$ (in other words $\pi$ is generically \'etale). The simplest proof of this theorem is by choosing the generators in the above form for a given $z\in Z$, then {to} verify that $h_z$ divides each maximal minor (in fact a little more is true, see  \cite{CCT21}) and {to} conclude using the fact that $h_z$ and $h_{z'}$ have no common factor if $z\not= z'$. The improvement from $3(d-1)$ comes from the fact that maximal minors provide a syzygy of the $f_i$'s via the Euler formula, thus a {GCD} $F$ of high degree will provide a syzygy of small degree for the $f_i$'s.

\begin{ques}
	At this moment, we are not aware of an example with $\sum_{z\in Z} \deg (h_z)>2d-2$, and wonder if this can hold or not.
\end{ques}

Theorem \ref{thm:jmat} can be seen as a particular case of the following result for a rational map $\psi$ from $\PP_k^m$ to $\PP_k^n$, defined on the complement $\Omega_\psi$ of the base locus of $\psi$. It could be used to detect subvarieties in $\PP_k^m$ that are contracted to lower dimensional ones by $\psi$ (see \cite[2.3]{CCT21})~:

\begin{prop}\label{Proposition1.3} 
Suppose that $V$ is a subvariety of $\PP_k^m$ such that $V\cap \Omega_\psi\neq \emptyset$ and let $r:=\dim V-\dim \psi(V)$. Then $V\subset V(I_{m-r+2}(J(f)))$, where $I_{m-r+2}(J(f))$ is the ideal generated by the $(m-r+2)$-minors of $J(f)$.
\end{prop}

\section{When the base locus is of positive dimension}\label{sec:orthoproj}

\def\ux{{\xi}}
\def\t{{\overline{t}}}
\def\u{{\overline{u}}}
\def\v{{\overline{v}}}
\def\dg{{\mathbf{d}}}
\def\Cc{{\mathcal{C}}}
\def\Oc{{\mathcal{O}}}
\def\Rc{{\mathcal{R}}}
\def\Bc{{\mathcal{B}}}
\def\ag{{\bm{a}}}
\newcommand\Fk{\mathfrak F}
\newcommand\Lk{\mathfrak L}
\def\Rees{{\mathrm{Rees}}}
\def\Proj{{\mathrm{Proj}}}
\def\Sym{{\mathrm{Sym}}}
\def\mug{{\bm{\mu}}}
\def\nug{{\bm{\nu}}}
\def\mg{{\bm{m}}}
\def\EE{{\mathbb{E}}}

In this section, we consider the case of a parameterization whose source is of {dimension three}. This problem has been recently considered in \cite{BBCY20} in order to compute the orthogonal projections of a point in space onto a parameterized algebraic surface. The main idea is to consider the congruence of the normal lines of the surface. Indeed, given a rational surface in {$\PP^3_k$} parameterized by $X$, its congruence of normal lines $\Psi$ is a rational map from $X\times \PP^1_k$ to $\PP^3_k$ and the orthogonal projections of a point $p\in \PP^3_k$ on $X$ are in correspondence with the pre-images of $p$ via $\Psi$. In comparison with the previous cases where the source was of dimension one or two, here the base locus may have a one-dimensional component.  In what follows we review the results obtained in \cite{BBCY20} with a particular focus on the new techniques that {are} used to tackle this new difficulty. It will also provide us the opportunity to illustrate how to work with blowup algebras over multigraded rings.

Thus, we consider a homogeneous parameterization 
\begin{eqnarray}\label{eq:defPsi}
	\Psi: X\times\iP^1_k & \dto & \iP^3_k \\ \nonumber
	       \ux \times(\t:t) & \mapsto &
		   \left({\Psi_1: \Psi_2: \Psi_3:  \Psi_4 } \right),
\end{eqnarray}
where $X$ stands for the spaces $\iP^2_k$ or $\iP^1_k\times\iP^1_k$ over an algebraically closed field $k$, and the $\Psi_i$'s are homogeneous polynomials in the coordinate ring of $X\times \iP^1_k$. The coordinate ring $R_X$ of $X$ is equal to $k[w,u,v]$ or $k[\u,u;\v,v]$, respectively, depending on $X$. The coordinate ring of $\iP^1$ is denoted by $R_1=k[\t,t]$ and hence the coordinate ring of $X\times \iP^1_k$ is the polynomial ring 
$R:=R_X\otimes_k R_1.$
{The polynomials $\Psi_1,\Psi_2,\Psi_3,\Psi_4$} are {hence multihomogeneous} polynomials of {the} same degree $(\dg,e)$, where $\dg$ refers to the degree with respect to $X$, which can be either an integer $d$ if $X=\iP^2_k$, or either a pair of integers $(d_1,d_2)$ if $X=\iP^1_k\times\iP^1_k$. 

\subsection{The base locus} We assume that $\Psi$ is a dominant map. We denote by $I$ the ideal of $R=R_X\otimes R_1$ generated by the defining polynomials of the map $\Psi$, i.e. {$I:=(\Psi_1,\Psi_2,\Psi_3,\Psi_4)$}. The irrelevant ideal of $X\times\iP^1_k$ is denoted by $B$; it is equal to the product of ideals $(w,u,v)\cdot (\t,t)$ if $X=\iP^2_k$, or to the product $(\u,u)\cdot (\v,v)\cdot (\t,t)$  if $X=\iP^1_k\times \iP^1_k$. The notation $I^{\textrm{sat}}$ stands for the saturation of the ideal $I$ with respect to the ideal $B$, i.e.~$I^{\textrm{sat}}=(I:B^\infty)$. 

The base locus of $\Psi$ is the subscheme of $X\times \iP^1_k$ defined by the ideal $I$; it is denoted by $\Bc$. Without loss of generality, $\Bc$ can be assumed to be of dimension at most one, but the presence of a curve component is a possibility after factoring out the gcd of the $\Psi_i$'s. When $\dim(\Bc)=1$ we denote by $\Cc$ its top unmixed one-dimensional curve component. We will need the following definition.
\begin{defn}\label{def:negsec} The curve $\Cc\subset X \times \iP^1_k$ {has} \emph{no section in degree $<(\ag,b)$} if for any $\bm{\alpha}<\ag$ and $\beta<b$, 
$H^0(\Cc,\mathcal{O}_\Cc(\bm{\alpha},\beta))=0$. 
\end{defn}

\subsection{Fibers} As we already mentioned previously, a proper definition of the fiber of a point under $\Psi$ requires to consider the graph of $\Psi$ and its closure $\Gamma \subset X\times \iP^1_k\times \iP^3_k$. Thus, the fiber of a point $p\in \iP^3_k$ is the subscheme 
\begin{equation}\label{eq:fiberRees}
	\Fk_p:=\Proj(\Rees_R(I)\otimes \kappa(p)) \subset X\times \iP^1_k,
\end{equation}
where $\kappa(p)$ denotes the residue field of $p$. As the equations of the Rees algebra $\Rees_R(I)$ are {in general} very difficult to get we also consider the corresponding symmetric algebra $\Sym_R(I)$ of the ideal $I$ and hence, as a {variation of \eqref{eq:fiberRees} we introduce} the subscheme 
\begin{equation}\label{eq:linearfiber}
\Lk_p:=\Proj(\Sym_R(I)\otimes \kappa(p)) \subset X\times \iP^1_k
\end{equation}
that we call \textit{the linear fiber of $p$}. We emphasize that the fiber $\Fk_p$ is always contained in the linear fiber $\Lk_p$ of a point $p$, and that they coincide if the ideal $I$ is locally a complete intersection at $p$ (see Proposition \ref{prop:graphci}).

\subsection{The main theorem} Theorem \ref{specizero} can be used to analyze finite linear fibers of $\Psi$ following the ideas we introduced in previous sections. In particular, a similar analysis of the regularity of these fibers can be done but there is an additional difficulty that is appearing if there exists a curve component $\Cc$ in the base locus $\Bc$. In order to describe the {multidegrees} for which the matrices $\MM_{(\mug,\nu)}$ (see Theorem \ref{specizero}) of the map $\Psi$ {yield a representation of} its finite fibers, we introduce the following notation. 

\begin{nt}\label{not:E} Let $r$ be a positive integer. For any $\bm{\alpha}=(\alpha_1,\ldots,\alpha_r) \in (\ZZ \cup \{ -\infty\} )^r$ we set 
	$$\EE (\bm{\alpha }):=\{ \bm{\zeta} \in \ZZ^r \ \vert \ \zeta_i \geq \alpha_i \textrm{ for all } i=1,\ldots,r\}.$$ 
It follows that, for any $\bm{\alpha}$ and $\bm{\beta}$ in $(\ZZ \cup \{ -\infty\} )^r$, $\EE (\bm{\alpha})\cap \EE (\bm{\beta})=\EE (\bm{\gamma})$ where $\gamma_i =\max\{ \alpha_i ,\beta_i \}$ for all $i=1,\ldots,r$, i.e.~ $\bm{\gamma}$ is the maximum of $\bm{\alpha}$ and $\bm{\beta}$ component-wise.	
\end{nt}

\begin{thm}[{\cite[Theorem 8]{BBCY20}}]\label{thm:regstabilization} Assume that we are in one of the two following cases:
	\begin{itemize}
		\item[(a)] The base locus $\Bc$ is finite, possibly empty,
		\item[(b)] $\dim(\Bc)=1$, $\Cc$ has no section in degree $<(\bm{0},e)$ and locally at every point $\iq \in \Proj(R)=X\times \PP^1_k$, the ideal $I_\iq$ is generated by at most three elements.
		
  	\end{itemize}
	Let $p$ be a point in $\iP^3_k$ such that $\Lk_p$ is finite, then
	$$ \corank \, \MM_{(\mug,\nu)}(p) =\deg(\Lk_p)$$
	for any degree $(\mug,\nu)$ such that
	\begin{itemize}
		\item if $X=\iP^2_k$,  
			$(\mu,\nu)\in \EE(3d-2,e-1) \cup \EE(2d-2,3e-1).$
		\item if $X=\iP^1_k\times \iP^1_k$, 
			$$(\mug,\nu) \in \EE(3d_1-1,2d_2-1,e-1) \cup \EE(2d_1-1,3d_2-1,e-1) 
			\cup \EE(2d_1-1,2d_2-1,3e-1).$$
	\end{itemize}	
\end{thm}

\subsection{Idea of the proof of the main theorem} If the base locus $\Bc$ of $\Psi$ is composed of finitely many points, then the proof of Theorem \ref{thm:regstabilization} goes along the same lines as the usual strategy developed in \cite{BuJo03,BC03,Bot10}.  However, if there exists a curve component in $\Bc$ then an additional difficulty {appears}. Indeed, if $\dim(\Bc)=0$ then $H^2_B(H_i)=0$ for all $i$, where $H_j$ denotes the homology module of the Koszul complex {$K_\bullet$} of the $\Psi_i$'s over $R$ and $H^i_B(-)$ the local cohomology modules with respect to the irrelevant $B$. If $\dim(\Bc)=1$ then it is necessary to control the multidegree at which these homology modules vanish. Following \cite{BBCY20}, we describe the main steps and tools to determine such multidegrees.

\begin{prop}[{\cite[Proposition 10]{BBCY20}}]\label{prop:vanishHiBHj} For any integer $i$, let $\Rc_i \subseteq \ZZ^r$ be a subset satisfying  
	\begin{equation*}
\forall j\in \ZZ : \Rc_i \cap \Supp (H^{j}_B(K_{i+j}))=\emptyset.		
	\end{equation*}
Then, if $\dim \Bc \leq 1$ the following properties hold for any integer $i$:
\begin{itemize}
	\item For all $\mug\in \Rc_{i-1}$, $H^1_B(H_i)_{\mug}=0$.
	\item There exists a natural graded map $\delta_i : H^0_B (H_i)\rightarrow H^2_B (H_{i+1})$ such that $(\delta_i)_\mug $ is surjective for all $\mug\in\Rc_{i-1}$ and is injective 
	for all $\mug\in\Rc_{i}$.
	\end{itemize}	
In particular, 
$$
H^0_B(H_i)_\mug \simeq H_B^2(H_{i+1})_\mug  \textrm{ for all } \mug\in \Rc_{i-1}\cap \Rc_{i}.
$$
\end{prop}

The regions $\Rc_i$ are obtained by the computation of cohomology of a product of projective spaces that is as follows in our case (recall that the Koszul modules $K_{i+j}$ are direct sums of shifted copies of $R$) :

\begin{lem}[{\cite[\S 6]{Bot10}}]\label{lem:locCohmoR} First, $H^i_B(R)=0$ for all $i\neq 2,3,4$. In addition, if $X=\PP^2_k$ then $R_X=k[w,u,v]$ and 
$$H^2_B(R)\simeq R_X\otimes\check{R_1}, \ \ H^3_B(R)\simeq \check{R_X}\otimes{R_1}, \ \ H^4_B(R)\simeq \check{R_X}\otimes\check{R_1}$$
where $\check{R_1}=\frac{1}{\t t}k[\t^{-1},t^{-1}]$ and $\check{R_X}=\frac{1}{wuv}k[w^{-1},u^{-1},v^{-1}]$.

If $X=\PP^1_k\times\PP^1_k$ then $R_X=R_2\otimes R_3$, where $R_2=k[\u,u]$, $R_3=k[\v,v]$, and 
$$H^2_B(R)\simeq\bigoplus_{\substack{\{i,j,k\}=\{1,2,3\},\\ j<k}} \check{R_i}\otimes R_j\otimes R_k,$$ 
$$H^3_B(R)\simeq\bigoplus_{\substack{\{i,j,k\}=\{1,2,3\}, \\j<k}} R_i\otimes \check{R_j}\otimes \check{R_k},\ \ H^4_B(R)\simeq \check{R_1}\otimes \check{R_2}\otimes \check{R_3}$$
where $\check{R_2}$ and $\check{R_3}$ are defined similarly to $\check{R_1}$.
\end{lem}

For the control of vanishing degrees of $H^2_B(H_i)$, a key ingredient is Serre duality. To be more precise, we have the following lemma that we state in a little more generality, when $\PP=\iP^{n_1}_k\times\cdots\times\iP^{n_r}_k$ is a product of projective spaces.

\begin{lem}[{\cite[Lemma 13]{BBCY20}}]\label{lem:H2BH2} Assume that $\dim(\Bc)=1$ and that the $s+1$ forms {$\Psi_1,\ldots,\Psi_{s+1}$} are of the same degree $\bm{\delta}$. Let $\Cc$ be the unmixed curve component of $\Bc$ and set  $p:=s-\dim \PP +2$ and $\bm{\sigma}:=(s+1)\bm{\delta} -(n_1 +1,\cdots ,n_r +1)$.
Then, for all $\mug\in\ZZ^r$, 
$$H^2_B(H_p)_\mug \simeq H^0(\Cc,\Oc_\Cc(-\mug+\bm{\sigma}  ))^\vee.$$
In particular, if $\Cc$ has no section in degree $<\mug_0$, for some $\mug_0 \in \ZZ^r$,
	then $$H^2_B(H_p)_\mug =0 \textrm{ for all } \mug \in \EE ((s+1)\bm{\delta}-(n_1 ,\cdots ,n_r )-\mug_0).$$
\end{lem}
\begin{proof}  As locally at a closed point $x\in \PP$, the $\Psi_i$'s contain a regular sequence of length $s-1$, and of length $s$ unless $x\in \Cc$, by \cite[\S 1-3]{BH} there are isomorphisms 
	$$\widetilde{H_p(\bm{\sigma}) }\simeq \widetilde{\ext^{s-1}_S (S/I,\omega_S)}\simeq \widetilde{\ext^{s-1}_S (S/I_\Cc ,\omega_S)}\simeq\omega_\Cc $$
from which we deduce that
\begin{equation}\label{eq:H2BwC}
	H^2_B(H_p)\simeq \bigoplus_{\mug} H^1( \Cc  ,\omega_\Cc (\mug -\bm{\sigma})).
\end{equation}
Now, applying Serre's duality Theorem \cite[Corollary 7.7]{Hart} we get
$$H^1(\Cc ,\omega_\Cc (\mug -\bm{\sigma}))\simeq H^0(\Cc,\Oc_\Cc(-\mug +\bm{\sigma}))^\vee,$$
which concludes the proof.
\end{proof}

\begin{lem}[{\cite[Lemma 14]{BBCY20}}]\label{lem:H2BH1} In the setting of Lemma \ref{lem:H2BH2},  let $s=\dim \PP$ and let $I'$ be an ideal generated by $s$ general linear combinations of the $\Psi_i$'s. If $I^{\textrm{sat}}=I'^{\textrm{sat}}$ then for all $\mug \in \ZZ^r$ there exists an exact sequence 
$$
 H^0(\Cc,\Oc_\Cc(-\mug -\bm{\delta}+\bm{\sigma}))^\vee\rightarrow H^2_B(H_{1})_\mug \rightarrow H^2_B(S/I)_{\mug -\bm{\delta}}\rightarrow 0.
 $$
 
In particular,  if $\Cc$ has no section in degree $<\mug_0$, for some $\mug_0 \in \ZZ^r$, then
$$
H^2_B(H_{1})_\mug =0,\quad \forall \mug \in \EE (s\bm{\delta}-(n_1 ,\cdots ,n_r )-\mug_0)\cap (\bm{\delta} +\Rc_{-2} ).
$$
\end{lem}

\begin{proof} We will denote by $H_i'$ the $i^{\mathrm{th}}$ homology module of the Koszul complex associated to $I'\subset R$. By \cite[Corollary 1.6.13]{BH} and \cite[Corollary 1.6.21]{BH} we have the following graded exact sequence
\begin{equation}\label{eq:exactseq}
0 \rightarrow M \rightarrow H_{1}' \rightarrow H_{1} \rightarrow H_{0}'(-\bm{\delta}) \rightarrow N \rightarrow 0
\end{equation}
with the property that the modules $M$ and $N$ are supported on $V(B)$, which implies that $H^i_B(M)=H^i_B(N)=0$ for $i\geq 1$.

This implies that the sequence
$$
 H^2_B(H_1')\rightarrow H^2_B(H_1) \rightarrow H^2_B(H_0')(-\bm{\delta}) \rightarrow  0 
 $$ is exact. 
 By Lemma \ref{lem:H2BH2}, $H^2_B(H_1')_\mug \simeq  H^0(\Cc,\Oc_\Cc(-\mug -\bm{\delta}+\bm{\sigma}))^\vee$ and one verifies that {the equalities} $H^2_B(H_0')(-\bm{\delta})_\mug =H^2_B(S/I)(-\bm{\delta})_\mug=0$ {hold} for {all} $\mug -\bm{\delta} \in \Rc_{-2}$. 
\end{proof}

From here, the proof of Theorem \ref{thm:regstabilization} follows {by the usual consideration} of the
{C}ech-Koszul spectral sequences associated to the approximation complex of cycles of $I$ and comparison between cohomology of Koszul cycles and homologies (see \cite[\S 4.2]{BBCY20} for more details).

\subsection{Curve with no section in negative degree} To apply Theorem \ref{thm:regstabilization} it is necessary that the curve component in the base locus, if any, has no section in negative {degrees}. Therefore, we now discuss when such a property holds.
   
A reduced irreducible scheme of positive dimension in a projective space has no section in negative degrees, this is due to the fact that the section ring is finite over $R$ and has no non-zero nilpotent element. Over a product of projective spaces it is typically not the case that the section ring is finitely generated, unless the scheme is a product of projective schemes. However, for instance using Veronese-Segre embeddings, one can easily show that it has no section in degrees $<0$ (all degree is strictly negative), which is sufficient for several applications.

An interesting question is anyhow to understand in which multidegrees a scheme could have sections in a product of projective schemes, and a closely related question (equivalent for schemes satisfying $S_2$) is to determine in what twists the top cohomology of the canonical module is not zero. Another related issue is to understand, for a projective scheme, if it has sections in negative degrees and what is the geometric meaning of these. In this direction, we reproduce a result (together  with a proof) showing that symbolic powers of prime ideals determines schemes with no sections in negative degrees (unless it is of dimension zero).

\begin{lem}Let $k$ be a field, $\Cc$ a geometrically reduced curve in $\iP^n_k$ and $t>0$. Then, the natural map
$$
H^0(\iP^n_k , \cO_{\Cc^{(t)}} (\nu))\ra H^0(\iP^n_k , \cO_{\Cc^{(t-1)}} (\nu ))
$$
is injective for $\nu < t-1$. In particular {$H^0(\iP^n_k , \cO_{\Cc^{(t)}} (\nu))=0$} for $\nu< 0$ and {$H^0(\iP^n_k , \cO_{\Cc^{(t)}})=H^0(\iP^n_k , \cO_{\Cc})$}.
\end{lem}

\begin{proof} This is clear for $t = 1$. Let $t \geq 2$. Write $I$ for the defining ideal of $\Cc$ and
$\omega_{A/I^{(j)}}:= \Ext^{n-1}_A(A/I^{(j)}, \omega_A)$. For any $\nu$, $(\omega_{A/I^{(j)}})_\nu =H^0(\iP^n_k , \omega_{\Cc^{(j)}}(\nu ))$ and setting $-^{\vee}:=\Hom_k (-,k)$,
$$
H^0(\iP^n_k , \cO_{\Cc^(j)} (\nu))=H^1(\iP^n_k , \omega_{\Cc^{(j)}}(-\nu ))^{\vee}=H^2_{\im}(A/I^{(j)})_{-\nu}^{\vee}.
$$
Consider the exact sequence
$$
\xymatrix{
0\ar[r]& \omega_{A/I^{(t-1)}}\ar[r]& \omega_{A/I^{(t)}}\ar[r]& \Ext^{n-1}_A (I^{(t-1)}/I^{(t)},\omega_A)\ar^{\psi}[r]&\Ext^n_A (A/I^{(t-1)},\omega_A).\\}
$$
As $\Ext^n_A (A/I^{(t-1)},\omega_A)$ has finite length and the other three modules are Cohen-Macaulay of dimension two, it gives rise to {an} exact sequence,
$$
\xymatrix{
0\ar[r]&{\rm im}(\psi ) \ar[r]&H^2_\im (\omega_{A/I^{(t-1)}})\ar[r]& H^2_\im (\omega_{A/I^{(t)}})\ar[r]&H^2_\im ( \Ext^{n-1}_A (I^{(t-1)}/I^{(t)},\omega_A))\ar[r]&0\\}
$$
and it remains to show that $H^2_\im ( \Ext^{n-1}_A (I^{(t-1)}/I^{(t)},\omega_A))_\nu =0$ for $\nu >1-t$.

First notice that $\Ext^{n-1}_A (I^{(t-1)}/I^{(t)},\omega_A)\simeq \Ext^{n-1}_A (I^{t-1}/I^{t},\omega_A)\simeq \Ext^{n-1}_A (\Sym^{t-1}_A(I/I^{2}),\omega_A)$, as $I$ is generically a complete intersection. There is an exact sequence 
$$
\xymatrix{
0\ar[r]&K \ar[r]&I/I^{2}\ar^(.4){\delta}[r]& A/I[-1]^{n+1}\ar[r]&\Omega_{A/I}\ar[r]&0\\}
$$
where $K$ is supported on the locus where $\Cc$ is not a complete intersection. Furthermore, locally on the smooth locus of $\Cc$, $\delta$ is split injective.  One deduces an exact sequence,
$$
\xymatrix{
0\ar[r]&K_t \ar[r]&\Sym^{t-1}_{A/I}(I/I^{2})\ar^(.3){\delta_t}[r]&\Sym^{t-1}_{A/I}(A/I[-1]^{n+1})=A/I[-t+1]^{{n+t-1}\choose{n}},\\}
$$
with $K_t$ supported on the non complete intersection locus of $\Cc$ and $\coker (\delta_t)$ of dimension two. This in turn gives an exact sequence
$$
(\omega_{A/I} [t-1])^{{n+t-1}\choose{n}}\ra \Ext^{n-1}_A (\Sym^{t-1}_{A/I}(I/I^{2}),\omega_A)\ra \Ext^{n}_A (\coker (\delta_t),\omega_A).
$$
As $\Ext^{n}_A (\coker (\delta_t),\omega_A)$ is of dimension at most 1, it follows that the natural map $$H^2_\im((\omega_{A/I} [t-1]))^{{n+t-1}\choose{n}}\ra H^2_\im(\Ext^{n-1}_A (\Sym^{t-1}_{A/I}(I/I^{2}),\omega_A))$$ is onto. On the other hand, $H^2_\im(\omega_{A/I})_\nu \simeq H_0(\iP^n_k , \cO_{\Cc} (-\nu))^{\vee}=0$ for $\nu >0$ as $\Cc$ is reduced.
Therefore $H^2_\im(\Ext^{n-1}_A (\Sym^{t-1}_{A/I}(I/I^{2}),\omega_A))_\nu =0$ for $\nu >1-t$, and the result follows.
\end{proof}

As, by Bertini theorem, the general hyperplane section  $Y=X\cap H$  of a (geometrically) reduced scheme $X$ is a (geometrically) reduced scheme of dimension one less and $Y^{(t)}=X^{(t)}\cap H$, the exact sequence 
$$
\xymatrix{
0\ar[r]&H^0(\iP^n_k , \cO_{X^{(t)}} (\nu -1))\ar[r]&H^0(\iP^n_k , \cO_{X^{(t)}} (\nu))\ar[r]&H^0(\iP^n_k , \cO_{Y^{(t)}} (\nu))}
$$
gives by induction on the dimension~:
\begin{thm}
If $X$ is a geometrically reduced scheme with all irreducible components of positive dimension and $t>0$, then
$$
H^0(\iP^n_k , \cO_{X^{(t)}} (\nu))=0, \quad \forall \nu <0,
$$
and $H^0(\iP^n_k , \cO_{X^{(t)}})=k$ if $X$ is equidimensional and connected in codimension one.
\end{thm}

In the case $X$ is irreducible and locally a complete intersection, and $k$ has characteristic zero, the above result follows from the generalization of Kodaira vanishing proved in \cite[Theorem 1.4]{BBL19}~: $H^\ell(\iP^n_k , \cO_{X^{(t)}} (\nu))=0$ for all $\nu<0$  and $\ell<\mathrm{codim}(\mathrm{Sing}(X))$.


\begin{thebibliography}{BBCY20}

\bibitem[BBC14]{BBC14}
Nicol\'{a}s Botbol, Laurent Bus\'{e}, and Marc Chardin.
\newblock Fitting ideals and multiple points of surface parameterizations.
\newblock {\em J. Algebra}, 420:486--508, 2014.

\bibitem[BBCY20]{BBCY20}
Nicol{\'a}s Botbol, Laurent Bus{\'e}, Marc Chardin, and Fatmanur Yildirim.
\newblock {Fibers of multi-graded rational maps and orthogonal projection onto
  rational surfaces}.
\newblock {\em {SIAM Journal on Applied Algebra and Geometry}}, 4(2):322--353,
  2020.

\bibitem[BBL{\etalchar{+}}19]{BBL19}
Bhargav Bhatt, Manuel Blickle, Gennady Lyubeznik, Anurag~K. Singh, and Wenliang
  Zhang.
\newblock Stabilization of the cohomology of thickenings.
\newblock {\em Amer. J. Math.}, 141(2):531--561, 2019.

\bibitem[BC03]{BC03}
Winfried Bruns and Aldo Conca.
\newblock Gr\"obner bases and determinantal ideals.
\newblock In {\em Commutative algebra, singularities and computer algebra
  ({S}inaia, 2002)}, volume 115 of {\em NATO Sci. Ser. II Math. Phys. Chem.},
  pages 9--66. Kluwer Acad. Publ., Dordrecht, 2003.

\bibitem[BC17]{BotCh}
Nicol\'{a}s Botbol and Marc Chardin.
\newblock Castelnuovo {M}umford regularity with respect to multigraded ideals.
\newblock {\em J. Algebra}, 474:361--392, 2017.

\bibitem[BCD03]{BCD03}
Laurent Bus{\'e}, David Cox, and Carlos D'Andrea.
\newblock Implicitization of surfaces in {${\PP}\sp 3$} in the presence of base
  points.
\newblock {\em J. Algebra Appl}, 2(2):189--214, 2003.

\bibitem[BCJ09]{BCJ06}
Laurent Bus{\'e}, Marc Chardin, and Jean-Pierre Jouanolou.
\newblock Torsion of the symmetric algebra and implicitization.
\newblock {\em Proc. Amer. Math. Soc}, 137(6):1855--1865, 2009.

\bibitem[BCS10]{BCS10}
Laurent Bus{\'e}, Marc Chardin, and Aron Simis.
\newblock Elimination and nonlinear equations of {R}ees algebras.
\newblock {\em J. Algebra}, 324(6):1314--1333, 2010.
\newblock With an appendix in French by Joseph Oesterl{\'e}.

\bibitem[BH93]{BH}
Winfried Bruns and J{\"u}rgen Herzog.
\newblock {\em Cohen-{M}acaulay rings}, volume~39 of {\em Cambridge Studies in
  Advanced Mathematics}.
\newblock Cambridge University Press, Cambridge, 1993.

\bibitem[BJ03]{BuJo03}
Laurent Bus{\'e} and Jean-Pierre Jouanolou.
\newblock On the closed image of a rational map and the implicitization
  problem.
\newblock {\em J. Algebra}, 265(1):312--357, 2003.

\bibitem[Bot11]{Bot10}
Nicol{\'a}s Botbol.
\newblock Implicit equation of multigraded hypersurfaces.
\newblock {\em J. Algebra}, 348(1):381--401, 2011.

\bibitem[CCT21]{CCT21}
Marc Chardin, Steven~Dale Cutkosky, and Quang~Hoa Tran.
\newblock Fibers of rational maps and jacobian matrices.
\newblock {\em Journal of Algebra}, 571:40--54, 2021.

\bibitem[CGZ00]{CGZ00}
David Cox, Ronald Goldman, and Ming Zhang.
\newblock On the validity of implicitization by moving quadrics of rational
  surfaces with no base points.
\newblock {\em J. Symbolic Comput.}, 29(3):419--440, 2000.

\bibitem[Cha07]{Cha07}
Marc Chardin.
\newblock Some results and questions on castelnuovo-mumford regularity.
\newblock {\em Lecture Notes in Pure and Applied Mathematics}, 254:1, 2007.

\bibitem[Cha13]{Cha12}
Marc Chardin.
\newblock Powers of ideals and the cohomology of stalks and fibers of
  morphisms.
\newblock {\em Algebra and Number Theory}, 7(1):1--18, 2013.

\bibitem[CJR13]{CJR}
Marc Chardin, Jean-Pierre Jouanolou, and Ahad Rahimi.
\newblock The eventual stability of depth, associated primes and cohomology of
  a graded module.
\newblock {\em J. Commut. Algebra}, 5(1):63--92, 2013.

\bibitem[CLO98]{CLO98}
David Cox, John Little, and Donal O'Shea.
\newblock {\em Using algebraic geometry}, volume 185 of {\em Graduate Texts in
  Mathematics}.
\newblock Springer-Verlag, New York, 1998.

\bibitem[CLS11]{CoxTV}
David Cox, John Little, and Henry Schenck.
\newblock {\em {Toric varieties.}}
\newblock {Providence, RI: American Mathematical Society (AMS)}, 2011.

\bibitem[Cox01]{Cox01}
David Cox.
\newblock Equations of parametric surfaces via syzygies.
\newblock {\em Contemporary Mathematics}, 286:1--20, 2001.

\bibitem[Cox20]{CoxCBMS}
David Cox.
\newblock {\em Applications of Polynomial Systems}, volume 134.
\newblock CBMS Regional Conference Series in Mathematics, 2020.

\bibitem[Dix09]{Dixon}
A.~L. Dixon.
\newblock The eliminant of three quantics in two independent variables.
\newblock {\em Proceedings of the London Mathematical Society}, s2-7(1):49--69,
  1909.

\bibitem[EH00]{EH00}
David Eisenbud and Joe Harris.
\newblock {\em {The geometry of schemes.}}
\newblock {Graduate Texts in Mathematics. 197. New York, NY: Springer. x, 294
  p.}, 2000.

\bibitem[Far97]{Farinbook}
Gerald Farin.
\newblock {\em Curves and surfaces for computer-aided geometric design}.
\newblock Computer Science and Scientific Computing. Academic Press, Inc., San
  Diego, CA, fourth edition, 1997.

\bibitem[GKZ94]{GKZ94}
Israel Gel{\cprime}fand, Mikhail Kapranov, and Andrei Zelevinsky.
\newblock {\em Discriminants, resultants, and multidimensional determinants}.
\newblock Mathematics: Theory \& Applications. Birkh{\"a}user Boston Inc,
  Boston, MA, 1994.
  
\bibitem[GLP83]{GLP83}  
Laurent Gruson, Robert Lazarsfeld and Christian Peskine. 
\newblock On a theorem of Castelnuovo, and the equations defining space curves. 
\newblock \emph{Invent. Math.} 72:491--506, 1983.   

\bibitem[Har77]{Hart}
Robin Hartshorne.
\newblock {\em Algebraic geometry}.
\newblock Springer-Verlag, New York, 1977.
\newblock Graduate Texts in Mathematics, No. 52.

\bibitem[HSV83]{HSV}
J{\"u}rgen Herzog, Aron Simis, and Wolmer~V Vasconcelos.
\newblock Koszul homology and blowing-up rings.
\newblock In {\em Commutative algebra (Trento, 1981)}, volume~84 of {\em
  Lecture Notes in Pure and Appl. Math}, pages 79--169. Dekker, New York, 1983.

\bibitem[Jou80]{Jou80}
Jean-Pierre Jouanolou.
\newblock Id\'eaux r\'esultants.
\newblock {\em Adv. in Math.}, 37(3):212--238, 1980.

\bibitem[Jou97]{Jou97}
Jean-Pierre Jouanolou.
\newblock Formes d'inertie et r\'esultant: un formulaire.
\newblock {\em Adv. Math.}, 126(2):119--250, 1997.

\bibitem[Lvo96]{Lvo96}
Sergei L'vovsky. 
\newblock On inflection points, monomial curves, and hypersurfaces containing projective curves. 
\newblock {Math. Ann.} 306:719--735, 1996. 

\bibitem[Mac02]{Mac02}
Francis Macaulay.
\newblock Some formulae in elimination.
\newblock {\em Proc. London Math. Soc.}, 33(1):3--27, 1902.

\bibitem[Mic64]{Mic64}
Artibano Micali.
\newblock Sur les alg\`ebres universelles.
\newblock {\em Ann. Inst. Fourier}, 14(2):33--87, 1964.

\bibitem[MS04]{MlS04}
Diane Maclagan and Gregory~G. Smith.
\newblock Multigraded {C}astelnuovo-{M}umford regularity.
\newblock {\em J. Reine Angew. Math.}, 571:179--212, 2004.

\bibitem[PM10]{Patbook}
Nicholas Patrikalakis and Takashi Maekawa.
\newblock {\em Shape interrogation for computer aided design and
  manufacturing}.
\newblock Springer-Verlag, Berlin, 2010.

\bibitem[SBAD16]{Shen:2016:LNS:3045889.3064444}
Jingjing Shen, Laurent Bus{\'e}, Pierre Alliez, and Neil Dodgson.
\newblock A line/trimmed nurbs surface intersection algorithm using matrix
  representations.
\newblock {\em Comput. Aided Geom. Des.}, 48(C):1--16, 2016.

\bibitem[Tei77]{Tei77}
Bernard Teissier.
\newblock The hunting of invariants in the geometry of discriminants.
\newblock In {\em Real and complex singularities ({P}roc. {N}inth {N}ordic
  {S}ummer {S}chool/{NAVF} {S}ympos. {M}ath., {O}slo, 1976)}, pages 565--678.
  Sijthoff and Noordhoff, Alphen aan den Rijn, 1977.

\bibitem[ZCG98]{Goldman}
Ming Zhang, Eng-Wee Chionh and Ronald Goldman.
\newblock Hybrid Dixon resultants.
\newblock {\em The Mathematics of Surfaces}, 8:193--212, 1998.

\end{thebibliography}

\newcommand{\etalchar}[1]{$^{#1}$}
\def\cprime{$'$}

\end{document}